\documentclass[11pt,twoside]{article}

\usepackage{fullpage}

\input{command-template}
\allowdisplaybreaks

\newcommand\blfootnote[1]{%
  \begingroup
  \renewcommand\thefootnote{}\footnote{#1}%
  \addtocounter{footnote}{-1}%
  \endgroup
}

\begin{document}

\begin{center}

  {\bf{\LARGE{Multiscale replay: A robust algorithm for stochastic variational inequalities with a Markovian buffer}}}

\vspace*{.2in}

{\large{
\begin{tabular}{ccc}
Milind Nakul$^{\star}$, Tianjiao Li$^{\ddagger,\diamond}$, Ashwin Pananjady$^{\star, \dagger}$
\end{tabular}
}}
\vspace*{.2in}

\begin{tabular}{c}
Georgia Tech, Schools of Industrial and Systems Engineering$^\star$ \& Electrical and Computer Engineering$^\dagger$ \\
MIT, Sloan School of Management$^\ddagger$
\end{tabular}

\vspace*{.15in}

\today

\vspace*{.15in}

\begin{abstract}
\blfootnote{$^\diamond$Tianjiao Li was at Georgia Tech when most of this work was performed.}
We introduce the Multiscale Experience Replay (MER) algorithm for solving a class of stochastic variational inequalities (VIs) in settings where samples are generated from a Markov chain and we have access to a memory buffer to store them. Rather than uniformly sampling from the buffer, MER utilizes a multi-scale sampling scheme to emulate the behavior of VI algorithms designed for independent and identically distributed samples, overcoming bias in the de facto serial scheme and thereby accelerating convergence. Notably, unlike standard sample-skipping variants of serial algorithms, MER is robust in that it achieves this acceleration in iteration complexity whenever possible, and without requiring knowledge of the mixing time of the Markov chain. 
We also discuss applications of MER,
particularly in policy evaluation with temporal difference learning and in training generalized linear models with dependent data.
\end{abstract}
\end{center}

\section{Introduction}
Variational inequalities (VIs) provide a general mathematical framework for modeling a wide range of problems in various domains, such as game theory, economics, operator theory, as well as modern machine learning \citep{browder1966existence, rockafellar1969convex, bach2008convex, chambolle2011first, ben2009robust, goodfellow2020generative}. An appealing aspect is that they allow for a unified treatment of both optimization and equilibrium problems. 

Given a closed convex feasible region $X \subseteq \mathbb{R}^n$ and an operator $F : X \mapsto \mathbb{R}^n$, the VI 
problem is to find an $x^* \in X$ such that
\begin{align}
    \label{eq:GMVI}
    \langle F(x^*),x-x^* \rangle \geq 0. \quad \text{ for all }  x \in X.
\end{align}
Throughout this paper, we assume the existence of a solution $x^*$ to problem \eqref{eq:GMVI}.
We also assume that $F$ is an $L$-Lipschitz continuous operator, i.e.,
for some $L>0$,
\begin{align}\label{eq:F_lipschitz}
    \| F(x_1)-F(x_2) \| \leq L\| x_1 - x_2 \|, \quad \text{ for all } x_1,x_2 \in X.
\end{align}
The notation $\| \cdot \|$ refers to the Euclidean norm throughout this paper, unless stated otherwise.

Additionally, we assume $F$ to be a generalized strongly monotone operator, i.e., for some $\mu > 0$, 
\begin{align}\label{eq:monotonicity}
    \langle F(x),x-x^* \rangle \geq \mu \|x-x^*\|^2, \quad \text{ for all } x \in X.
\end{align} 

Clearly, condition \eqref{eq:monotonicity} holds if $F$ is strongly monotone, i.e., 
\begin{align*}
\langle F(x_1) - F(x_2),x_1-x_2 \rangle \geq \mu \|x_1-x_2\|^2, \quad \text{ for all } x_1, x_2 \in X.
\end{align*}
Note that our assumed generalized strong monotonicity condition \eqref{eq:monotonicity} is weaker than global strong monotonicity, but already guarantees that $x^*$ satisfying Ineq.~\eqref{eq:GMVI} is unique.

The wide range of applications of VIs has motivated extensive algorithmic studies in the past few decades. Classical algorithms for monotone VIs include (but are not limited to) the gradient projection method \citep{ sibony1970methodes, bertsekas1997nonlinear}, the extragradient method \citep{korpelevich1976extragradient}, and the proximal point algorithm \citep{rockafellar1976monotone}. Roughly two decades ago~\citet{nemirovski2004prox} introduced the mirror-prox algorithm, a breakthrough that inspired numerous studies for solving VI problems in both deterministic \citep{auslender2005interior,nesterov2006solving, nesterov2007dual,  monteiro2010complexity, malitsky2015projected, dang2015convergence} and stochastic settings \citep{juditsky2011solving, nemirovski2009robust,chen2017accelerated, iusem2017extragradient, yousefian2017smoothing,cui2021analysis, kotsalis2022simple, tianjiao2022}. 

\subsection{Stochastic VIs with dependent data}
In this work, we operate in the latter stochastic setting where only inexact information about the operator $F$ is available.
Specifically, we assume the existence of a stochastic oracle, which, when given any query point $x \in X$, generates a stochastic estimator $\widetilde{F}(x,\xi) \in \mathbb{R}^n$, where $\xi \sim \pi$ is a random variable defined in some probability space $\Xi$. The estimator satisfies the property that at any fixed point $x$, it is marginally unbiased, in that
\begin{align}\label{eq:expectation_eq}
    F(x) = \EE[\widetilde{F}(x,\xi)] = \int _{\xi \in \Xi} \widetilde{F}(x,\xi) d \pi(\xi).
\end{align}
A typical assumption in the stochastic optimization literature is that we have access to a stream of samples $(\xi_1,\xi_2,...)$ that are drawn i.i.d. from $\pi$. However, this assumption is violated in important real-world scenarios such as Markov decision processes (MDPs) and reinforcement learning (RL) \citep{sutton2018reinforcement}.
In such settings it is common to confront the more challenging \emph{Markovian noise} setting, where $(\xi_t: t = 1, 2, \ldots)$ is a Markov process defined on the state space $\Xi$, and  $\pi$ denotes the unique stationary distribution of this process. 
Stochastic VIs with Markovian noise have recently garnered significant attention across applications in machine learning, statistics, optimization, and control, where data often exhibits serial correlations \citep{harold1997stochastic, spall2005introduction,benveniste2012adaptive,  meyn2012markov}. The key challenge in this setting is that sample correlations lead to \emph{biased} estimators of the operator $F$. Concretely,  letting $\mathcal{F}_{t - 1}$ denote the $\sigma$-field generated by the history $(\xi_1, \ldots, \xi_{t - 1})$, we have $\EE[\widetilde{F}(x,\xi_t) | \mathcal{F}_{t-1}] \neq F(x)$. The conditional bias $\EE[\widetilde{F}(x,\xi_t) | \mathcal{F}_{t-1}] - F(x)$ introduces challenges for developing fast-converging algorithms. 

Having said that, stochastic optimization with dependent data has seen significant study. While we cannot hope to cover this vast literature, we refer the reader to \citet{duchi2012ergodic} and references therein for a discussion of classical developments.  
Specifically relevant to VIs with Markovian noise is the literature on policy evaluation in reinforcement learning, which is a fundamental and extensively researched problem.  
Here, the most popular algorithm is called temporal difference (TD) learning, and can be viewed as serial stochastic approximation with a suitably constructed VI. The algorithm was first introduced by \citet{sutton1988learning} and theoretical analysis of TD learning has seen significant progress over the years. \citet{tsitsiklis1996analysis} provided one of the earliest results, proving the asymptotic convergence of TD learning with function approximation. Other classical results on asymptotic convergence include those by \citet{borkar2000ode} and \citet{borkar2008stochastic}, who used an ordinary differential equation (ODE) framework to analyze the long-term behavior of TD learning. More recently, \citet{bhandari2018finite} established finite-time performance guarantees for TD learning.
Other notable contributions include works by \citet{srikant2019finite,chen2021lyapunov,durmus2021stability, min2021variance,li2024q,wu2025uncertainty}.  

Broadly speaking, guarantees in this literature decompose the error $\| x_T - x^* \|^2$ at iteration $T$ into two components: One component, which tracks the progress of the algorithm even in the absence of noise, is called the deterministic error and typically decays exponentially in $T$; there have been various methods proposed to accelerate the exponent in this convergence in terms of how it depends on the conditioning of the problem (see, e.g.,~\citet{li2023accelerated} and references therein). The other component, called the stochastic error, captures how stochasticity influences convergence---this is typically the slower error term in noisy settings and forms the focus of this paper. With this context in mind, let us more concretely describe two representative results in this literature in order to motivate our paper.

\begin{figure}
    \centering
    \includegraphics[width=0.6\linewidth]{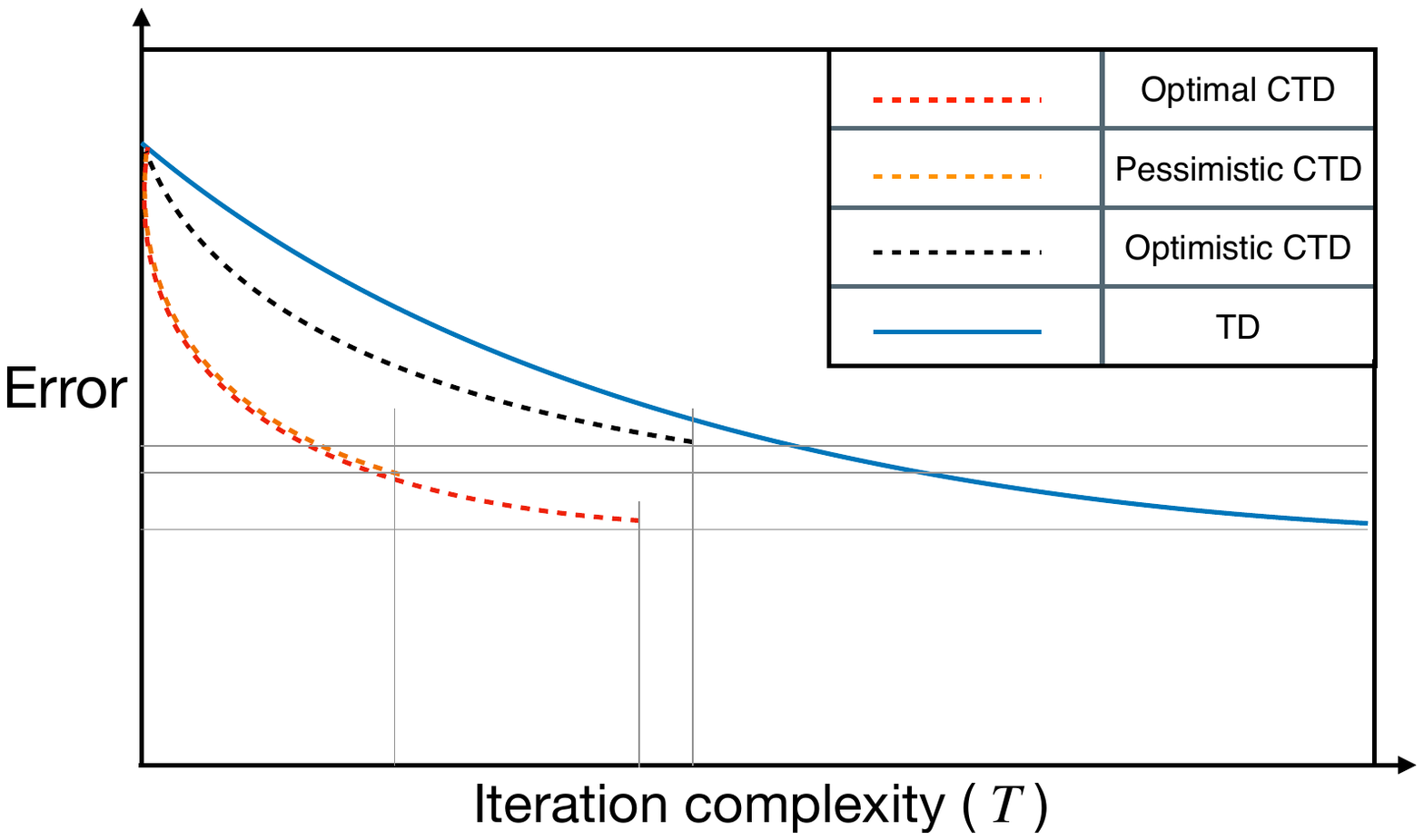} 
    \caption{ A schematic diagram showing the progression of error with the number of iterations when $T$ is the number of available samples. The blue curve represents the result of running the serial stochastic approximation algorithm, which we label as TD. The red curve represents the CTD algorithm~\citep{tianjiao2022} with the skipping parameter chosen correctly, with exact information about the mixing time. This curve has accelerated iteration complexity while achieving the same eventual error as TD. The orange curve represents a pessimistic CTD algorithm in which the skipping parameter is chosen to be much larger than the actual mixing time. Although this curve traces the optimal CTD curve for a few iterations, it stops much earlier and incurs large eventual error because it exhausts all available samples sooner. The black curve represents an optimistic version of CTD with the skipping parameter less than the mixing time. Although this runs for more iterations, it typically does not achieve an eventual error comparable to TD.}
    \label{fig:compare}
\end{figure}

On the one hand, for the problem of solving (projected) linear fixed point equations---which is a special case of VIs with a linear operator---\citet{mou2024optimal} recently analyzed the classical serial stochastic approximation algorithm on Markovian data with iterate averaging. They demonstrated a convergence rate of $\order(\bar{\tau} \sigma^2/T)$\footnote{Here $\sigma^2$ scales linearly in $d$, where $d$ is the dimension of the problem.} for the stochastic error, where $\sigma^2$ is a suitably defined noise level in the problem, $\bar{\tau}$ is the mixing time of the underlying Markov chain and $T$ is the total number of iterations; since the algorithm proceeds serially, this is also the total number of samples processed. On the other hand, \citet{tianjiao2022} introduced the conditional temporal difference (CTD) algorithm for solving the stochastic VI problem with Markovian data. They tackled the problem of data correlation via \emph{skipping} of samples and showed that the stochastic error after $T$ iterations could be bounded as $\order(\sigma^2/T)$, thereby accelerating iteration complexity and obtaining a guarantee comparable to one obtained for i.i.d. samples. However, CTD requires prior knowledge of the Markov chain's mixing properties to decide how many samples to skip per iteration; unfortunately, perfect knowledge of these  properties is rarely available a priori in real-world applications. This is also not a purely theoretical requirement; over- or under-specifying the skipping parameter during implementation of these algorithms can lead to undesirable convergence behavior, and a schematic of these issues is illustrated in Figure~\ref{fig:compare}. A key open question is thus to understand if the accelerated convergence of CTD can be attained \emph{without knowing} the mixing properties of the chain. Indeed, questions of this flavor---framed under the broad umbrella of adaptive and parameter-free algorithms---have received some recent interest in both optimization and RL~\citep[see, e.g.][]{dorfman2022adapting,ju2025auto}.

\subsection{Memory buffers and experience replay}

In this paper, we make use of an additional resource that is typically available in practice and could be helpful in designing algorithms that converge quickly without knowledge of mixing properties, which is a \emph{memory buffer}.
 Specifically, suppose that we can store the samples in a buffer of size $B$, where we start the algorithm using the vector of previously collected samples
\begin{align}
    \label{eq:markovian_buffer}
    \xi^B = (\xi_t)_{t=1}^B,
\end{align}
and update this buffer as (and if) we gather more samples. Serial stochastic approximation algorithms such as TD simply process all samples in the buffer in causal order, while skipping-based algorithms such as CTD aim to process every $\overline{\tau}$ samples in the buffer.
Can we instead strategically select samples from the buffer at each iteration of the algorithm? Note that since these samples are stored, the algorithm has the flexibility to process entries of the buffer in any order.

Memory buffers are commonly used in reinforcement learning (RL), and a very popular algorithm that arises from this paradigm is called  experience replay (ER), which was initially introduced by \citet{lin1992self}. Broadly speaking, most ER algorithms process the buffer by sampling data points uniformly at random,
and these techniques have gained widespread popularity for breaking temporal dependencies in data when training RL agents that gather samples on the fly~\citep{mnih2015human,wang2016sample, zhang2017deeper, fedus2020revisiting, kumar2022look}.
Despite its empirical success, the theoretical understanding of ER in iterative methods is still in its infancy~\citep[see, e.g.][]{nagaraj2020least, agarwal2021online, kowshik2021streaming, lazic2021improved}. For example, \citet{agarwal2021online} derived sample complexity bounds for Q-learning using the super-martingale structure of ``reverse experience replay". \citet{kowshik2021streaming} proposed an algorithm using reverse experience memory to address linear system identification problems, while \citet{lazic2021improved} provided regret bounds for regularized policy iteration with replay buffers in the context of average-reward MDPs. Additionally, \citet{pmlr-v162-di-castro22a} offered insights into the stochastic properties of the replay buffer, establishing the asymptotic convergence of actor-critic algorithms in RL. 

While ER has been largely confined to settings in RL, these ideas have not been more broadly adopted in stochastic optimization. Furthermore, the question of \emph{how} to optimally sample entries of the buffer in order to achieve optimal rates of convergence has not been thoroughly investigated to our knowledge. This paper aims to fill this important gap in the context of developing VI algorithms, and to apply these insights to specific examples of VIs (including in RL).

\subsection{Contributions and organization}
Our main contribution is to propose and analyze a new algorithm, which we call \emph{multiscale experience replay} (MER), for stochastic optimization in VIs with Markovian data when we have access to a memory buffer. 
In more detail:
\begin{enumerate}
    \item Our algorithm involves a careful sampling scheme to select entries of the memory buffer with respect to which algorithmic updates are performed. Specifically, the algorithm operates in epochs, and within each epoch, the iterates are updated using samples separated by a replay gap that decreases over epochs. Importantly, the implementation of this scheme does not require any information about the mixing time of the Markov chain. Under standard assumptions, we prove iteration complexity bounds on the algorithm that show acceleration comparable to skipping samples optimally, as in the CTD algorithm.
    \item In addition to establishing upper bounds on the error, we also show that in its initial epochs, our algorithm \emph{emulates} the behavior of the analogous iteration run on i.i.d. data provided the chain mixes in $1$-Wasserstein distance and the stochastic operator is appropriately regular. While it has been informally argued that experience replay enjoys such an advantage, we rigorously establish a two-sided bound showing that the error of our algorithm is tracked by its analog run on i.i.d. samples.
    \item Finally, we apply MER to generalized linear models as well as the policy evaluation problem in RL. For both problems, we show convergence rates that are comparable to state of the art guarantees but now with an algorithm that is agnostic to the mixing time of the chain. Analogous results regarding emulation of the i.i.d. setting are also established.
\end{enumerate}
The remainder of this paper is organized as follows. In Section \ref{subsec:examples_intro}, we provide some illustrative examples of VIs that we focus on in the paper. Section \ref{sec:MER} introduces the MER algorithm for VIs. In Section \ref{sec:theoretical_results}, we present our assumptions and main results in the general setting. In Section \ref{sec:examples}, we discuss consequences of these results for our illustrative examples. In Section~\ref{sec:numerical_results}, we present  numerical experiments that corroborate our theoretical results, and Section \ref{sec:proofs} collects proofs of all our claims. We conclude the paper with a discussion of open problems in Section~\ref{sec:discussion}.

\section{Illustrative examples}
\label{subsec:examples_intro}
Recall from our earlier set-up that we are interested in finding a point $x^*$ in a closed convex feasible region $X$ such that the condition in Ineq.~\eqref{eq:GMVI} is satisfied, when the samples are stored in a Markovian buffer of the form~\eqref{eq:markovian_buffer}. In this section, we describe two concrete examples of interest. 
\subsection{Generalized linear models (GLMs)}
\label{sec:GLM}
Consider the following generalized linear regression model --- also called the single-index model ---  
where we have an unknown signal $x^* \in X$ and observations are generated according to
\begin{align}
    \label{eq:GLM_model}
    y_t = f(a_t^\top x^*) + v_t.
\end{align}
Here $a_t \in \mathbb{R}^d$ denotes the $t$-th sample of covariates, $y_t \in \mathbb{R}$ denotes the corresponding response, and the scalar $v_t$ denotes unobserved zero-mean i.i.d. noise such that $\EE[v_t^2] = \sigma_{v}^2$. Suppose $f:\mathbb{R} \to \mathbb{R}$ is a Lipschitz continuous and strongly monotone link function satisfying, for all $x, y$ in the domain of $f$, the pair of inequalities $|f(x) - f(y)| / | x - y | \leq L_f$ and $(f(x) - f(y)) (x - y) \geq \mu_f (x - y)^2$. Monotone GLMs model several applications spanning data analysis and signal recovery~\citep[see, e.g.][]{juditsky2020statistical,lou2025accurate,pananjady2021single, ilandarideva2024accelerated, zhu2025beyond,haque2025stochastic}. 

Now suppose that the covariates $(a_t)_{t = 1}^B$ are generated from an ergodic Markov chain with finite mixing time such that the $a_t$'s remain bounded almost surely, i.e., $\|a_t\| \leq D_a$ a.s. for all $t$. Moreover, suppose that $a_1$ is drawn from the stationary distribution $\pi$ of the chain, and let $\Lambda:= \EE[a_ta_t^\top]$ denote the stationary covariance satisfying $\Lambda \succcurlyeq \kappa I$ for some $\kappa > 0$.

For each sample covariate $a_t$, the corresponding observation $y_t$ is obtained using the GLM model from Eq.~\eqref{eq:GLM_model}. To pose the signal recovery problem in our framework, note that the $t$-th sample is given by the tuple $\xi_t = (a_t, y_t)$, and our Markovian buffer collects $B$ samples $(a_t, y_t)_{t = 1}^B$. Our goal is to estimate the underlying signal $x^*$ using samples from this buffer.
In order to solve for $x^*$, we formulate a stochastic VI problem with the operators defined as
\begin{align}
    \label{eq:glm_definition}
    F(x) = \EE_{ a \sim \pi}\left[a f(a^Tx)\right]-\EE_{a \sim \pi}\left[a f(a^{T}x^*)\right] \text{ and } \widetilde{F}(x,\xi_t) = a_t f(a_t^Tx)-a_t y_t.
\end{align}
By definition, note that 
 $F(x^*) = 0$, which implies that the VI condition in Ineq.~\eqref{eq:GMVI} is satisfied.

\paragraph{Verifying conditions on operator $F$.} Note that
\begin{align*}
    \|F(x) - F(y)\| &= \|\EE_{a\sim\pi}\left[af(a^Tx) - af(a^Ty)\right]\| \\ 
    &\leq \EE_{a\sim \pi}\left[ \|a\||f(a^Tx) - f(a^Ty)| \right] \\ 
    &\overset{\1}{\leq}  L_f\EE_{a\sim \pi}\left[ \|a\| \|a^T(x-y)\| \right] \\ 
    &\overset{\2}{\leq} L_f\EE_{a\sim \pi} \left[ \|a\|^2 \right] \|x-y\|  \\
    &\leq L_f \cdot D_a^2 \cdot \|x-y\|,
\end{align*}
where step $\1$ follows from the Lipschitzness of the link function $f$ and step $\2$ follows from the Cauchy--Schwarz inequality. Thus, the operator $F$ is Lipschitz with parameter $L_f \cdot D_a^2$.

We also have
\begin{align*}
    \langle F(x),x-x^* \rangle &= \EE_{a\sim \pi} \left[ \langle af(a^T x) -af(a^Tx^*),x-x^* \rangle \right] \\ 
    &= \EE_{a\sim \pi} \left[ (f(a^Tx)- f(a^Tx^*) a^T(x-x^*)) \right]\\
    &\overset{\1}{\geq} \mu_f \EE_{a\sim \pi} \left[  (a^T(x-x^*))^2 \right] \\ 
    &= \mu_f (x-x^*)^T \EE_{a\sim \pi}[aa^T] (x-x^*) \overset{\2}{\geq} \mu_f \cdot \kappa \cdot \|x-x^*\|^2,
\end{align*}
where step $\1$ follows from the strong monotonicity of $f$ and step $\2$ follows from the assumption $\Lambda \succcurlyeq \kappa I$. Thus, the operator $F$ is generalized strongly monotone with parameter $\mu_f \cdot \kappa$.

\subsection{Policy evaluation in reinforcement learning}
\label{sec:RL}
Policy evaluation is formulated using a Markov reward process (MRP), described by the tuple $(\mathcal{S},P,R,\gamma)$, where $\mathcal{S}$ denotes the state space, $P$ is the transition matrix, $R$ is the reward function and $\gamma \in (0,1)$ is the discount factor; see, e.g.,~\citet{puterman2014markov,bertsekas2019reinforcement} for more details and background. At each time step, the system moves from the current state $s \in \mathcal{S}$ to some new state $s' \in \mathcal{S}$ with probability $P(s,s')$ while the agent receives a reward of $R(s,s')$ corresponding to this transition. The expected instantaneous reward for state $s$ is defined as $r(s) = \sum_{s' \in S} R(s,s')P(s,s')$. The value function of a state is defined as the infinite-horizon discounted reward received starting in that particular state 
\begin{align*}
    v^{*}(s) = \EE\left[\sum_{t=0}^{\infty}\gamma^{t}R(s_t,s_{t+1})| s_{0}=s\right],
\end{align*}
and the goal of policy evaluation is to solve for $v^*$.
Letting $D = |\mathcal{S}|$ denote the number of states, the value function $v^{*}$ and the reward function $r$ are $D$-dimensional vectors of reals, and the following Bellman equation is satisfied
\begin{align}\label{eq:Bellman}
    v^{*}  = r + \gamma P v^{*}.
\end{align}
We assume that the Markov chain is ergodic, thus there is a unique stationary distribution $\pi$ with strictly positive entries satisfying $\pi P = \pi$. 
Let $\Pi := \text{diag}(\pi_1,\pi_2,...,\pi_{D}) \in \mathbb{R}^{D\times D}$.

In modern variants of this problem with large state spaces,
it is common to seek approximate solutions for the Bellman Eq.~\eqref{eq:Bellman}. Specifically, we choose a $d$-dimensional subspace, defined as $\mathbb{S} := \text{span}\{\psi_1,\hdots,\psi_d\}$ for $d$ linearly independent basis vectors and attempt to solve for $\overline{v} \in \mathbb{S}$ through the 
\emph{projected} Bellman equation, which can be written in matrix-vector notation as follows.

Defining the matrix $\Psi := [\psi_1,\ldots,\psi_d]^T$, the equation reads
\begin{align*}
    \Psi \Pi (\bar{v}-r-\gamma P\bar{v})=0,
\end{align*}
and since $\bar{v} := \Psi^T \bar{\theta}$ for some $\overline{\theta}$, the projected Bellman equation can be equivalently written in $\mathbb{R}^d$ as
\begin{align}
    \label{eq:projected_bellman_matrix}
    \Psi \Pi \Psi^T \bar{\theta}=\Psi \Pi r+\Psi \Pi\gamma P\Psi^T \bar{\theta}.
\end{align}
For simplicity, we define a set of orthonormal basis vectors for the subspace $\mathbb{S}$, i.e., $\Phi \equiv [\phi_1,\ldots,\phi_d] := Q^{-1/2}\Psi$ where $Q_{i,j} := \langle \psi_i,\psi_j \rangle$. Denote $ \eigmax (Q)$ and $\eigmin(Q)$ as the maximum and the minimum eigenvalues of the matrix $Q$, respectively. We can solve the projected Bellman Eq.~\eqref{eq:projected_bellman_matrix} by viewing it as a VI with the following operators:
\begin{align*}
    F(\theta) = \Psi \Pi (\Psi^T \theta- r-\gamma P\Psi^T \theta) \text{ and } \widetilde{F}(\theta, \xi_k) = \left( \left\langle \psi(s_k),\theta \right\rangle - R(s_k,s_{k+1})-\gamma \left\langle\psi(s_{k+1}),\theta  \right\rangle \right)\psi(s_k).
\end{align*}
 From the definition of the operator $F$ above, we observe that for $\bar{\theta}$, which is the solution to the projected Bellman Eq.~\eqref{eq:projected_bellman_matrix}, $F(\bar{\theta}) = 0$. This means that the VI condition in Ineq.~\eqref{eq:GMVI} is satisfied for $\bar{\theta}$. 

In the Markovian setting, we are given a trajectory from the Markov reward process \\ \mbox{$(s_t,s_{t+1},R(s_t,s_{t+1}))_{t\geq0}$}, where $s_0$ is drawn from the stationary distribution.
From these, we form a Markovian buffer containing samples
\begin{align}
    \label{eq:RL_operator}
    \xi^{B} = ( \xi_t = (\psi(s_t),\psi(s_{t+1}),R(s_t,s_{t+1})) )_{t=1}^B.
\end{align}
Our goal is to estimate $\bar{\theta}$ using samples from the Markovian buffer.

\paragraph{Verifying conditions on operator $F$.} It can be readily verified that $F$ is a $(1+\gamma)\cdot \lambda_{max}(Q)$-Lipschitz continuous and $(1-\gamma)\cdot \lambda_{min}(Q)$-strongly monotone operator \cite[see, e.g.][]{li2023accelerated}.

\section{Multiscale experience replay}
\label{sec:MER}
Having provided examples that motivate our work, we now introduce our main algorithm for solving stochastic VIs with a Markovian buffer. 
As a stepping stone to our procedure, recall that standard stochastic approximation (SA) with sample-skipping (termed CTD in~\citet{tianjiao2022}) is given by
\begin{align}\label{eq:TD_update}
    x_{t+1} = \arg \min_{x \in X}\left\{ \eta \langle \widetilde{F}(x_{t},\xi_{t \bar\tau}),x \rangle + \frac{1}{2}\|x_{t}-x\|^2\right\}, \text{ for } t=1,2,\ldots .
\end{align}
This update uses samples separated by $\bar\tau$ steps, a value that should ideally be chosen proportional to the mixing time of the Markov chain. However, as previously mentioned and illustrated in Figure~\ref{fig:compare}, misspecifying the skipping parameter $\bar{\tau}$ can lead to both statistical and computational inefficiencies.

To resolve this issue, we propose the following multiscale experience replay (MER) algorithm.
Given a memory buffer of size $B$, the MER algorithm judiciously chooses elements of this buffer by keeping the ``active'' samples used in the SA updates as far apart as possible. The method is formally presented as Algorithm~\ref{alg:two}. 

\begin{algorithm}[H]  \caption{Multiscale Experience Replay}  
		\label{alg:two}   
		\begin{algorithmic} 
			\STATE{\textbf{Input}: Memory Buffer $\{\xi_1,\xi_2,\hdots,\xi_{B}\}$, Total number of epochs $K$.}
			\FOR{$ k = 1, \ldots, K$}
			\STATE{Re-initialize the variable $\initialiterate \in X$ independently of the past.}
                \STATE{Define the replay gap $\tau_k= \frac{B}{2^k}$ and the sample size used in $k$-th epoch $T_k = 2^k$.}
			\FOR{$t=1, \ldots, T_k$}
				\STATE{Using the sample $\xi_{t\tau_k}$, perform one step of SA:
        \begin{align}\label{eq:MER_update}
             \updatediterate = \arg \min_{x\in X} \eta_k \langle \widetilde{F}(\iterate,\xi_{t\tau_k}),x \rangle+\frac{1}{2} \|x-\iterate\|^2.
         \end{align}}
                \STATE{Delete the sample $\xi_{t\tau_k}$ and append the new sample to the end of the  buffer.}
			\ENDFOR
			\ENDFOR
		\end{algorithmic}
	\end{algorithm}

Operationally, the memory buffer is utilized over multiple passes (epochs), where we re-initialize the iterate at the beginning of each epoch. In epoch $k$, the replay gap, defined as $\tau_k = B / 2^k$, sets the separation between samples used to update the iterate within that epoch. The total number of samples used in the $k$-th epoch is thus $T_k = 2^k$.
The replay gap $\tau_k$ progressively decreases the scale at which the algorithm operates. If $\tau_k =\bar \tau$ for some $k$, then the algorithm executes the skipped SA update~\eqref{eq:TD_update} in that epoch. This multiscale process continues until the final epoch recovers the standard serial SA update, which proceeds without any sample skipping.

Note that as stated, Algorithm \ref{alg:two} is designed for the online setting where we continue to observe new samples $\xi_{B+1}, \xi_{B+2}, \ldots$ as we run iterations and wish to store them in the buffer. At each iteration, the algorithm processes the current sample and removes it from the buffer, appending a newly observed sample to the end of the buffer. This algorithm can be easily adapted to an offline setting where no new samples are available by omitting the step that deletes the current sample. In this case, the buffer remains static, allowing the same data points to be reused across multiple epochs. All our results also extend to the static setting,

\section{Theoretical results}
\label{sec:theoretical_results}
In this section, we present convergence guarantees for the MER method proposed in Algorithm~\ref{alg:two}. 
We start by introducing and discussing the assumptions that underlie our analysis.
\subsection{Assumptions and shorthand notation}\label{subsec:assumptions}
We begin by making a mild assumption regarding the optimal solution $x^*$.
\begin{assumption}
    \label{assmp:optimal_ball}
    The optimal solution $x^*$ lies in an $\ell_2$ ball of known radius $D/2$, i.e. $x^* \in \mathbb{B}_2(D/2)$. Thus, we initialize iterates at the beginning of each epoch within this ball independently of the past, so that $\initialiterate \in \mathbb{B}_2(D/2)$.
\end{assumption}

We now make the following mild assumptions regarding the  stochastic operator $\widetilde{F}$, and we will verify in Section \ref{sec:examples} that these assumptions hold in the two specific examples that we introduced in Section~\ref{subsec:examples_intro}. 
\begin{assumption}
\label{assmp:lipscitz1}
    The stochastic operator $\widetilde{F}$ is a Lipschitz continuous map in its first argument, i.e., there exists a $\widetilde{L}_1 >0$, such that for any $\xi \in \Xi$, we have
\begin{align}\label{eq:lipscitz1}
    \|\widetilde{F}(x_1,\xi)-\widetilde{F}(x_2,\xi)\| \leq \widetilde{L}_1\|x_1-x_2\|.
\end{align}
\end{assumption}

\begin{assumption}[Variance of stochastic operator]
\label{assmp:state_dependent_noise}
    There exist positive scalars $\sigma, \zeta \geq 0$ such that for every point $x \in X$ and pair of natural numbers $t' \geq t$, we have with probability $1$:
\begin{align}\label{eq:state_dependent_noise}
    \EE\left[ \big\| \widetilde{F}(x,\xi_{t'})-\EE [ \widetilde{F}(x,\xi_{t'}) | \Fspace_{t-1}] \big\|^2 \big| \Fspace_{t-1} \right] \leq \frac{\sigma^2}{2} + \frac{\zeta^2}{2}\|x-x^*\|^2. 
\end{align}
\end{assumption}
To motivate our next two assumptions, note that if $\xi_{t}$ is drawn from the stationary distribution independently of $x$, then the stochastic operator $\widetilde F$ is an unbiased estimator of $F$, satisfying $\EE [ \widetilde{F}(x,\xi_{t'})] = F(x)$. However, since the noise is Markovian in our setting, the stochastic operator  is biased conditionally on the past, i.e., for $t' \geq t$, we have $\EE [ \widetilde{F}(x,\xi_{t'}) | \Fspace_{t-1}] \neq F(x)$.
Our next two assumptions ensure that the Markov chain mixes fast enough so that the bias at the optimal point $x^*$ can be properly controlled, and the conditional bias of the difference of two operator evaluations is also controlled.
\begin{assumption}[Bias at the optimal solution]
\label{assmp:optimal_condition}
    There exist constants $\cm>0$ and $\rho_1 \in (0,1)$ such that for every pair of natural numbers $t' \geq t$, the following inequality holds with probability 1:
\begin{align}\label{eq:optimal_condition}
    \|F(x^*)-\EE[\widetilde{F}(x^*,\xi_{t'})|\Fspace_{t-1}]\| \leq \cm \cdot \rho_1^{t' - t}.
\end{align}
\end{assumption}
\begin{assumption}[Bias of difference]
\label{assmp:mixing_norm_revised}
    There exist constants $\cp>0$ and $\rho_2 \in (0,1)$ such that for every pair of points $x,y \in X$ and pair of natural numbers $t' \geq t$, the following holds with probability $1$:   \begin{align}\label{eq:mixing_norm_revised}
    \|F(x)-\EE[\widetilde{F}(x,\xi_{t'})|\Fspace_{t-1}]-F(y)+\EE[\widetilde{F}(y,\xi_{t'})|\Fspace_{t-1}]\| \leq \cp \cdot \rho_2^{t' - t} \cdot \|x-y\|.
\end{align}
\end{assumption}
Given their close connection to the mixing properties of the Markov chain, we refer to $\rho_1$ and $\rho_2$ in the above assumptions as \emph{mixing parameters}.
We define $\rho \defn \max\{\rho_1,\rho_2\} \in (0,1)$, so that both Assumption \ref{assmp:optimal_condition} and Assumption~\ref{assmp:mixing_norm_revised} are satisfied with the common mixing parameter $\rho \in (0,1)$. We also define 
\begin{align}\label{eq:C_defn}
    C \defn \frac{\cm}{40}+\cp,
\end{align}
that combines the multiplicative mixing constants from Assumption~\ref{assmp:optimal_condition} and Assumption~\ref{assmp:mixing_norm_revised}.

Having stated our assumptions, we proceed to define some additional notation to facilitate the presentation of our main theorem and subsequent discussions.
We start by defining the following time scale parameter $\tau_{M}$, which characterizes the number of time steps needed for the bias of the operator $\widetilde F$ to be sufficiently reduced,
\begin{align}
    \label{eq:surrogate_mixing_time}
    \tau_{M} \defn \frac{\log ({18C}/{\mu})}{\log ({1}/{\rho})}.
\end{align}
This parameter is related to how fast the Markov chain mixes, and also involves the strong monotonicity scalar $\mu$. Throughout the paper, we refer to $\tau_{M}$ as the ``effective mixing time''. 
We also define 
\begin{align}\label{eq:alpha_definition}
    \alpha_k \defn \frac{\tau_M}{\tau_k}.
\end{align}
Note that the parameter $\alpha_k$ evolves with $k$ monotonically. To develop intuition, it is useful to divide this evolution into the following two stages:
\begin{itemize}
    \item ($\tau_k > \tau_M$) When the replay gap is greater than the effective mixing time, we have $\alpha_k <1$. Roughly speaking, the samples used in such an epoch $k$ should be thought of as nearly i.i.d. due to the mixing conditions in Assumption~\ref{assmp:optimal_condition} and Assumption~\ref{assmp:mixing_norm_revised}.
    \item ($\tau_k \leq \tau_M$) When the replay gap is less than the effective mixing time, we have $\alpha_k \geq 1$. The samples in this case are inevitably correlated, and this correlation becomes stronger as $\tau_k$ decreases.
\end{itemize}
Therefore, the parameter $\alpha_k$ encapsulates the shift from (almost) independent to correlated samples as the epochs advance. 

To define our final pieces of notation, recall the parameter $D$ from Assumption~\ref{assmp:optimal_ball} and define 
\begin{align}\nonumber
    &\bar{L} \defn L + \widetilde{L}_1,\;M \defn \max \left\{2+ \frac{4C}{\mu},\frac{40\cm}{\mu}\right\}\\ \label{eq:definitions}
    &p_k \defn 1+ (\log T_k)^{-1}\cdot\log\left(\frac{ \mu^2 M\max\{1,D^2\}}{3 (6\sigma^2+ 4 \|F(x^*)\|^2)}\right).
\end{align}
These parameters will simplify notation in our theorem statements. Operationally, the parameter $\bar{L}$ combines the Lipschitz constants of the operators $F$ and $\widetilde{F}$.

\subsection{Convergence guarantees for MER}\label{subsec:MER_convergence}
With the assumptions stated and parameters defined in Section \ref{subsec:assumptions}, we are in a position to provide convergence guarantees for the MER method.
\begin{theorem}\label{thm:main_error}
    Suppose Assumptions \ref{assmp:optimal_ball}-\ref{assmp:mixing_norm_revised} hold. If the buffer size satisfies the lower bound \\
    \mbox{$B \geq \frac{32 p_1 \tau_M(\zeta^2 + 16 \bar{L}^2 )}{3\mu^2}\log \left( \frac{32 p_1 \tau_M(\zeta^2 + 16 \bar{L}^2)}{3\mu^2} \right)$} and we have the step-size schedule 
    \begin{align}
    \label{eq:step_size}
        \eta_k \defn
\min \left\{\frac{3 \mu}{16(\zeta^2 + 16\bar{L}^2)},\frac{p_k\log T_k}{\mu T_k} \right\} \quad \text{ for } k = 1, \ldots, K,
    \end{align}
    then
    \begin{align}
    \nonumber
        \EE\|\finaliterate - x^*\|^2&\leq M \left(1+\frac{3 \mu^2}{8\left( \alpha_k+1 \right)(\zeta^2 + 16 \bar{L}^2)} \right)^{-T_k}\left(D^2+1  \right)\\ \nonumber 
        &\qquad+ \frac{20  \cm\rho^{\tau_M+\tau_k-1}}{ \mu}\left(1+\frac{3}{128\left( \alpha_k+1 \right)} \right) \\ 
        &\qquad + \frac{3(2+p_k\log T_k)}{\mu^2 T_k} \left(\alpha_k +1 \right)  (6\sigma^2+25\cm^2 \rho^{2(\tau_k-1)}+ 4 \|F(x^*)\|^2). \label{eq:main_bound}
    \end{align}
\end{theorem}
See Section~\ref{subsec:proof_theorem_1} for the detailed proof of this theorem.

A few remarks are in order. First, note that in order to set the step-size schedule in Eq.~\eqref{eq:step_size}, we require access to some parameters\footnote{There is some dependence on the pair $(C, C_M)$, but these quantities are not as critical for mixing as $\rho$ and $\tau_M$.} that depend on the population operator $F$, but crucially not on $\rho$ or $\tau_M$, which define the mixing properties of the chain. This stands in stark contrast to other algorithms in this literature~\citep[e.g.][]{tianjiao2022}, and reveals a key advantage of having access to a memory buffer as in our setting.

Let us next discuss the error terms in the guarantee~\eqref{eq:main_bound}. The first term on the RHS of Ineq. \eqref{eq:main_bound}, which converges linearly, corresponds to the deterministic error of the MER method -- this term would exist even in the absence of stochasticity in the problem. It is likely that the dependence on the condition number $\bar{L}/\mu$ in this term can be improved by applying the operator extrapolation technique proposed in \citet{kotsalis2022simple, tianjiao2022}; as mentioned before, obtaining optimal accelerated convergence of the deterministic error is beyond the scope of our paper. 
The second term on the RHS of Ineq.~\eqref{eq:main_bound} is due to the bias at the optimal point $x^*$. 
This term is typically negligible since it decays exponentially with $ \tau_{M}$.

As mentioned earlier, our focus is on the last term in the RHS of Ineq.~\eqref{eq:main_bound}, which is a bound on the stochastic error. Specifically, 
when $B$ is at least of the order $T_k \tau_M$, or equivalently when $\alpha_k = \mathcal{O}(1)$,
this stochastic error term simplifies to 
\begin{align} \label{eq:SE-bound}
\order\left(\frac{3(2+p\log T_k)}{ \mu^2T_k}  \left(6\sigma^2+ 4 \|F(x^*)\|^2\right)\right),
\end{align}
which nearly matches the stochastic error term in the setting where we process $T_k$ i.i.d. samples. 
Conversely, when $\alpha_k \gg 1$, then the stochastic error~\eqref{eq:SE-bound} is inflated by a multiplicative factor of $\alpha_k$, depending on the ratio of the effective mixing time and the replay gap. However, note that we always have $\alpha_k \leq \tau_M$, so this inflation factor can at most be on the order of the effective mixing time. In fact, except in the last epoch when $\alpha_k = \tau_M$, MER outperforms the standard Markovian SA in \cite{tianjiao2022} by a multiplicative factor of the replay gap $\tau_k$. We also recover the guarantees in  \citet{tianjiao2022} in the last epoch when the replay gap satisfies $\tau_k = 1$. 

\subsection{Tracking the i.i.d. error}
\label{sec:following_iid}
In the previous section, we demonstrated that when the replay gap $\tau_k$ is larger than the effective mixing time $\tau_M$, the MER algorithm has a rate of convergence resembling that of i.i.d. sampling with $T_k$ samples. While encouraging, this is simply a comparison of upper bounds and leaves open the question of whether the i.i.d. behavior is actually emulated by MER.  
In this section, we analyze if and how the error of the MER algorithm --- especially in the initial epochs --- tracks the error of running simple SA on an i.i.d. sample sequence. 

Before stating such a guarantee, we require some additional definitions and assumptions. Begin by defining the Wasserstein distance on a metric space, $(\Xi,r)$. For a pair $(\mu,\mu')$ of probability distributions on $\Xi$, let $\mathcal{P}(\mu,\mu')$ denote the space for all possible couplings of $\mu$ and $\mu'$. For any $q\geq1$, the Wasserstein-$q$ distance between $\mu$ and $\mu'$ is given by 
\begin{align}
    \label{defn:wasserstein_dist}
    \mathcal{W}_{q,r}(\mu,\mu') = \left\{ \inf_{\gamma \in \mathcal{P}(\mu,\mu')}\int_{\Xi \times \Xi}r(x,y)^q d\gamma(x,y) \right\}^{1/q}.
\end{align}

We take the metric $r$ to be the Euclidean distance throughout this paper, unless stated otherwise. 
We now make the assumption that the chain mixes fast in Wasserstein distance. 
\begin{assumption}[Wasserstein mixing]
    \label{assmp:wasserstein_mixing}
    With the effective mixing time $\tau_M$ defined in Eq.~\eqref{eq:surrogate_mixing_time}, for any $x,y \in X$, there exists universal positive constant $c_0$ such that
    \begin{align}
        \label{eq:wasserstein_mixing}
        \mathcal{W}_{1}(\delta_x P^{\tau_M},\delta_y P^{\tau_M}) \leq c_0 \|x-y\|,
    \end{align}
    where $\delta_x$ denotes the distribution that places all its mass on $x$ and recall that $P$ is the transition matrix of the Markov chain.
\end{assumption}
We also require an additional Lipschitzness assumption on the second argument of the stochastic operator $\widetilde{F}$.
\begin{assumption}\label{assmp:lipscitz2}
    The stochastic operator $\widetilde{F}$ is a Lipschitz continuous map in its second argument, i.e., there exists a positive constant $\widetilde{L}_2$, such that for any $x\in X$, we have
\begin{align}\label{eq:lipscitz2}
    \|\widetilde{F}(x,\xi_1)-\widetilde{F}(x,\xi_2)\| \leq \widetilde{L}_2\|\xi_1-\xi_2\|.
\end{align}
\end{assumption}

In the initial epochs, the replay gap $\tau_k$ is large. Concretely, to state our guarantee below, suppose that for some $\beta > 1$ we have $\tau_k = \beta \tau_M$.
Let $x_{T+1}$ denote the iterate obtained after applying $T$-steps of the MER update of Eq.~\eqref{eq:MER_update} at epoch $k$, starting from some initialization $x_1$. In particular, conditional on $x_1$, the randomness in $x_{T+1}$ arises purely from the Markovian samples $(\xi_{k \beta \tau_M})_{k=1}^{T}$ used in the updates. To compare this iterate with an i.i.d. analogue, define $\widetilde{x}_{T+1}$ as the iterate obtained by starting from the same initialization $x_1$ and applying $T$-steps of the update 
\[
\widetilde{x}_{t + 1} = \arg \min_{x \in X}\left\{ \eta \langle \widetilde{F}(\widetilde{x}_{t},\widetilde{\xi}_{t \beta \tau_M}),x \rangle + \frac{1}{2}\|\widetilde{x}_{t}-x\|^2\right\}, \text{ for } t=1,2,\ldots,
\]
where $(\widetilde{\xi}_{k \beta \tau_M})_{k=1}^{T}$ denotes an i.i.d. sequence of samples drawn from the stationary distribution $\pi$. Let us define the errors in the two cases as $\Delta_{T+1} \defn x_{T+1}-x^*$ and $\widetilde{\Delta}_{T+1} \defn \widetilde{x}_{T+1}-x^*$, respectively. 

\begin{theorem}\label{thm:emulate_iid}
Suppose Assumptions \ref{assmp:lipscitz1},\ref{assmp:wasserstein_mixing} and \ref{assmp:lipscitz2} hold. If the step-size parameter is chosen to satisfy $\eta \leq \frac{1}{\widetilde{L}_1 T}$, then 
    \begin{align}\label{eq:lipschitz_bound}
    \left|\EE\|\Delta_{T+1}\| - \EE\|\widetilde{\Delta}_{T+1}\|\right| \leq 3c_0 \cdot  \frac{ \widetilde{L}_2}{\widetilde{L}_1} \cdot \sqrt{T} \cdot e^{-\beta}.
\end{align}
Consequently, if $\beta = \log T^{3/2}$, then the difference between the algorithmic errors is lower order and satisfies
\begin{align}\label{corollary:emulate_iid}
    \left|\EE\|\Delta_{T+1}\| - \EE\|\widetilde{\Delta}_{T+1}\|\right| \leq  \frac{3c_0 \widetilde{L}_2}{T\widetilde{L}_1}.
\end{align}
\end{theorem}
See Section~\ref{sec:pf_thm_emulate_iid} for a proof of this theorem.

Qualitatively, Theorem \ref{thm:emulate_iid} demonstrates that when the replay gap $\tau_k$ is sufficiently large (specifically, when $\beta$ is large), the error of our MER algorithm on a Markovian data buffer closely resembles the error of standard SA on an i.i.d. sample sequence. This finding directly confirms the core benefit sought by any experience replay strategy---it is able to emulate i.i.d. behavior even with dependent data. As a consequence, as long as the SA algorithm on an i.i.d. sequence achieves a faster convergence rate than the worst-case Markovian bound, the MER algorithm will replicate this faster rate provided the epoch number $k$ is not too large. 

\section{Consequences for specific examples}
\label{sec:examples}
Having established our general results, we now specialize them to the examples of GLMs and policy evaluation discussed in Section~\ref{subsec:examples_intro}. In these examples, it is common to assume that the Markov chain mixes in total variation (TV) distance. The TV distance between a pair $(\mu',\mu)$ of probability distributions on $\Xi$ is defined as follows:
\begin{align}\label{def:TV_distance}
    \TV(\mu,\mu') = \frac{1}{2} \int_{\xi \in \Xi} |d\mu(\xi) - d\mu'(\xi)|.
\end{align}
For an ergodic Markov chain with unique stationary distribution $\pi$, we denote the conditional distribution of the random variable $\xi_{t'} \in \Xi$ at time step $t'$ conditioned on $\mathcal{F}_{t-1}$ as $P_{t-1}^{t'}$. Recall that $\mathcal{F}_t$ denotes the $\sigma$-algebra generated by the random variables $\{\xi_k\}_{k=1}^{t}$. A commonly used TV mixing assumption is the following:
\begin{assumption}
    \label{assmp:TV_mixing}
    For $t'\geq t$, there exists positive constants $C_{TV}$ and $\rho_{TV} \in (0,1)$ such that we have 
    \begin{align*}
        \TV(\pi,P_{t-1}^{t'}) \leq C_{TV} \cdot \rho_{TV} ^{t'-t} \quad \text{w.p. } 1.
    \end{align*}
\end{assumption}
\subsection{Generalized linear models}
Revisiting the GLM example from Section \ref{sec:GLM}, our objective is to estimate the underlying signal $x^*$ using the given Markovian buffer $\xi^B$. Theorem~\ref{thm:main_error} yields the following corollary for the GLM model. 

 \begin{corollary}\label{corollary:GLM_1}
    Suppose Assumption~\ref{assmp:optimal_ball} holds. In addition assume that the Markov chain generating observations $(a_t, y_t)_{t\geq 1}$ is geometrically mixing in the total variation metric so that it satisfies Assumption~\ref{assmp:TV_mixing}. Choosing step-sizes prescribed by Theorem \ref{thm:main_error}, 
    we have
    \begin{align*}
        \EE\|\finaliterate - x^*\|^2 \leq M \left(1+\frac{3 \mu_f^2 \kappa^2}{\left( \alpha_k+1 \right)8(\zeta^2 + 64 L_f^2 D_a^4)} \right)^{-T_k}\left(D^2+1  \right)  + \frac{18(2+p_k\log T_k)}{\mu_f^2 \kappa^2 T_k} \left(\alpha_k +1 \right)  \sigma^2.
    \end{align*}
    Here $\sigma^2 = 6 D_a^2 \sigma_v^2$ and $\zeta^2 = 12 L_f^2 D_a^4$.
\end{corollary}
We prove Corollary~\ref{corollary:GLM_1} in Section~\ref{sec:pf_corollary_GLM_1}.

 Comparing with prior results, we note that the deterministic error term in the bound converges at a linear rate, which matches the lower bound established in \citet{nagaraj2020least} [Theorem 1]. The dependence on the effective mixing time $\tau_M$ (through $\alpha_k$) is information-theoretically optimal for this term.
In terms of dependence on the dimension of the problem $d$, we observe that the parameter $\sigma^2$ typically scales linearly with $d$. Consequently, the stochastic error for GLMs for the MER algorithm is of the form
$$\widetilde{\order}\left(\frac{d\sigma_d^2}{T_k}\left(\frac{\tau_M}{\tau_k}+1\right)\right),$$
where $\sigma_d^2$ represents the dimension-independent noise level of the problem~\citep[c.f.][Theorem 3]{nagaraj2020least}.
Also, Corollary~\ref{corollary:GLM_1} implies that to find an $\epsilon$-optimal solution for signal estimation in GLMs for the $k$-th epoch (i.e., $\EE\|\finaliterate - x^*\|^2 \leq \epsilon$), the required iteration complexity is: $$\mathcal{O}\left(\max\left(\frac{(\alpha_k+1)\bar{L}^2}{\mu_f^2}\log\frac{(D^2+1)}{\epsilon},\frac{\sigma^2(\alpha_k+1)}{\epsilon \mu_f^2}\log \frac{1}{\epsilon}\right)\right).$$ This iteration complexity bound is nearly optimal, up to a logarithmic factor, in terms of dependence on $\epsilon$ \citep{tianjiao2022}.

We now turn to stating how well the MER algorithm tracks the i.i.d. error for the GLM model. We persist with the previous notation: $\Delta_{T+1}$ denotes the error of running $T$ steps of the MER algorithm on Markovian samples with replay gap $\tau_k = \beta \cdot \tau_M$ for $\beta = \log T^{3/2}$, while $\widetilde{\Delta}_{T+1}$ denotes the error of running $T$ steps of standard SA on i.i.d. samples. 

We additionally assume that $X$ has an $\ell_2$ diameter of at most $D_x$ and we have a bounded observation space, i.e., $|y_t| \leq D_y$ for all $t$. Since $f$ is Lipschitz and the covariates are also almost surely bounded, for some $D_f\geq 0$ we have
\begin{align}\label{eq:f_bound}
    \sup_{\substack{a: \| a \| \leq D_a \\ x \in X}} |f(a^\top x)| \leq D_f.
\end{align}

\begin{corollary}\label{corollary:GLM_2}
    Suppose that assumptions stated above hold. Additionally, suppose that the Markov chain generating observations $(a_t, y_t)_{t \geq 1}$ is mixing in the Wasserstein-$1$ metric so that Assumption~\ref{assmp:wasserstein_mixing} is satisfied. Under
    the parameter choice of Theorem \ref{thm:emulate_iid}, we have
    \begin{align*}
        \left|\EE\|\Delta_{T+1}\| - \EE\|\widetilde{\Delta}_{T+1}\|\right| \leq  \frac{3c_0 \max (L_f D_a D_x + D_f + D_y , D_a)}{TL_f D_a^2}.
    \end{align*}
\end{corollary}
We prove Corollary~\ref{corollary:GLM_2} in Section~\ref{sec:pf_corollary_GLM_2}. To our knowledge, this is the first result establishing closeness of SA iterates under the Markovian and i.i.d. settings for GLMs.

\subsection{Policy evaluation in RL}
Recall the example of policy evaluation in RL from Section \ref{sec:RL}, where we aim to find the fixed point $\bar{\theta}$ of the projected Bellman Eq.~\eqref{eq:projected_bellman_matrix}. Theorem~\ref{thm:main_error} yields the following corollary for this problem.

\begin{corollary}\label{corollary:RL_1}
    Suppose Assumption~\ref{assmp:optimal_ball} holds and the feature vectors are bounded, i.e., for all $s \in \mathcal{S}, \|\psi(s)\| \leq D_{\psi}$ and the reward function is bounded, i.e., for all $s,s' \in \mathcal{S}$, $R(s,s') \leq \bar{R}$. In addition, we assume that the underlying Markov chain is geometrically mixing in the total variation metric so that Assumption~\ref{assmp:TV_mixing} is satisfied. Under the same premise as that of Theorem \ref{thm:main_error},
    we have
    \begin{align*}
        \EE\|\theta_{T_k +1}^{(k)} - \bar{\theta}\|^2&\leq M \left(1+\frac{3 (1-\gamma)^2\lambda_{min}(Q)^2}{\left( \alpha_k+1 \right)8(\zeta^2 + 16 (1+\gamma)^2 (D_{\psi}^2 + \lambda_{max}(Q))^2)} \right)^{-T_k}\left(D^2+1  \right)\\ 
        &\qquad+ \frac{20  \cm\rho^{\tau_M+\tau_k-1}}{ (1-\gamma)\lambda_{min}(Q)}\left(1+\frac{3 }{128\left( \alpha_k+1 \right)} \right) \\
        &\qquad + \frac{3(2+p_k\log T_k)}{(1-\gamma)^2\lambda_{min}(Q)^2 T_k} \left(\alpha_k +1 \right)  (6\sigma^2+25\cm^2 \rho^{2(\tau_k-1)}+ 4 \|F(\bar{\theta})\|^2).
    \end{align*}
        Here $\sigma^2 = 48(1+\gamma)^2 D_{\psi}^4 D^2+ 48 D_{\psi}^2 \bar{R}^2$ and $\zeta^2 = 12 (1+\gamma)^2 D_{\psi}^4$.
\end{corollary}
We prove Corollary~\ref{corollary:RL_1} in Section~\ref{sec:pf_corollary_RL_1}.

Ignoring the second term as before, the sample complexity required to find an $\epsilon$-optimal solution for policy evaluation in RL, such that $\EE\left[\|\theta_{T_k+1}^{(k)}-\bar{\theta}\|^2\right]\leq \epsilon$ , during the $k$-th epoch is
$$\mathcal{O}\left(\max\left(\frac{\alpha_k+1}{(1-\gamma)^2}\log\frac{D^2+1}{\epsilon},\frac{\bar{\sigma}^2(\alpha_k+1)}{(1-\gamma)^2\epsilon}\log\frac{1}{(1-\gamma)^2\epsilon}\right)\right),$$
where $\bar{\sigma}^2=6\sigma^2+25\cm^2\rho^{2(\tau_k-1)}+4\|F(\bar{\theta})\|^2$. 
The first term in the complexity bound corresponds to the deterministic error, which exhibits a quadratic dependence on the discount factor $\frac{1}{(1-\gamma)^2}$, which is suboptimal. One can further apply the acceleration technique in \citet{li2023accelerated} to achieve the optimal dependence on the discount factor. The main quantity of interest in the complexity bound is the last term, i.e., the stochastic error, which is nearly optimal, up to a logarithmic factor, in terms of its dependence on $\epsilon$. The dependence on the discount factor can also be further improved by applying variance-reduction techniques; see, e.g., \citet{khamaru2021temporal, li2023accelerated}. The key takeaway, however, is that the given rate is achieved without knowledge of the mixing time.

The next corollary shows that the MER algorithm tracks the i.i.d. error in its initial epochs. 

\begin{corollary}\label{corollary:RL_2}
    Suppose that the underlying Markov chain is mixing in the Wasserstein-$1$ metric so that Assumption~\ref{assmp:wasserstein_mixing} is satisfied. Also suppose that the feature vectors are bounded, i.e., for all $s \in \mathcal{S}, \|\psi(s)\| \leq D_{\psi}$ and the reward function is bounded, i.e., for all $s,s' \in \mathcal{S}$ $R(s,s') \leq \bar{R}$, and consider the same parameter choice of Theorem \ref{thm:emulate_iid} and $\beta = \log T^{3/2}$.
    Then the following holds
    \begin{align}
        \nonumber
        \left|\EE\|\Delta_{T+1}\| - \EE\|\widetilde{\Delta}_{T+1}\|\right| \leq  \frac{3c_0 \left((2+\gamma)\cdot D_{\psi} \cdot D_{\theta} +\bar{R}\right)}{T(1+\gamma)D_{\psi}^2}.
    \end{align}
\end{corollary}
We prove Corollary~\ref{corollary:RL_2} in Section~\ref{sec:pf_corollary_RL_2}. As in the case of GLMs, we believe this is the first such result in the literature.

\section{Numerical results}\label{sec:numerical_results}

In this section, we report numerical experiments for the MER algorithm, comparing its performance to standard serial SA and SA with sample-skipping. We consider both the logistic regression problem and the policy evaluation problem. We also include results for the SA algorithm on i.i.d.\ data to assess how well each method can emulate the i.i.d.\ setting.

\subsection{Policy evaluation with Markovian noise}

We first consider the policy evaluation problem, where we compare the MER algorithm against TD and CTD, which are the analogs of SA for this setting without and with skipping, respectively.
 We consider Markov reward processes with a parameter $m$ that controls the mixing time of the underlying Markov chain, while the discount factor is fixed at $\gamma=0.8$ for all runs. Given a state space $\mathcal{S}$, the transition kernel $P$ is defined as
\begin{align*}
    P(s'|s) = \left \{
    \begin{array}{ll}
    \frac{2m-1}{m}, & s' = s\\
    \frac{1-m}{(|\mathcal{S}|-1)m}, & s' \neq s
    \end{array}
    \right..
\end{align*}
The reward function is $r(s) = \ind{s \geq d}$, where $d$ is the dimension of the feature vector. The stationary distribution for this chain is uniform over all states, and the spectral gap of $P$ is $\frac{|\mathcal{S}|(1-m)}{(|\mathcal{S}|-1)m}$. The effective mixing time, $\tau_M$, is therefore of order $\frac{m}{1-m}$: as $m$ approaches $1/2$, the chain mixes rapidly and approaches the i.i.d.\ regime, while larger $m$ values yield slower mixing and more pronounced Markovian dependence. When $m=1$, the chain is fully sticky, remaining in the same state with probability $1$.

Feature vectors are defined as
\[
    \phi(s) = \mathbf{e}_{s \wedge d} \in \mathbb{R}^d,
\]
where $\mathbf{e}_i$ is the $i$-th standard basis vector. We fix $|\mathcal{S}|=30$ and $d=16$. For each value of $m$, we plot the normalized squared error $\frac{\|\overline{\theta}_{T_k} - \theta^*\|^2}{\|\theta_0 - \theta^*\|^2}$, where $\overline{\theta}_{T_k}$ is the averaged output, $\theta^*$ is the optimal solution, and $\theta_0$ is the initialization.

\begin{figure}[htbp]
    \centering
    \subfigure[]{\label{fig:1}\includegraphics[width=8cm]{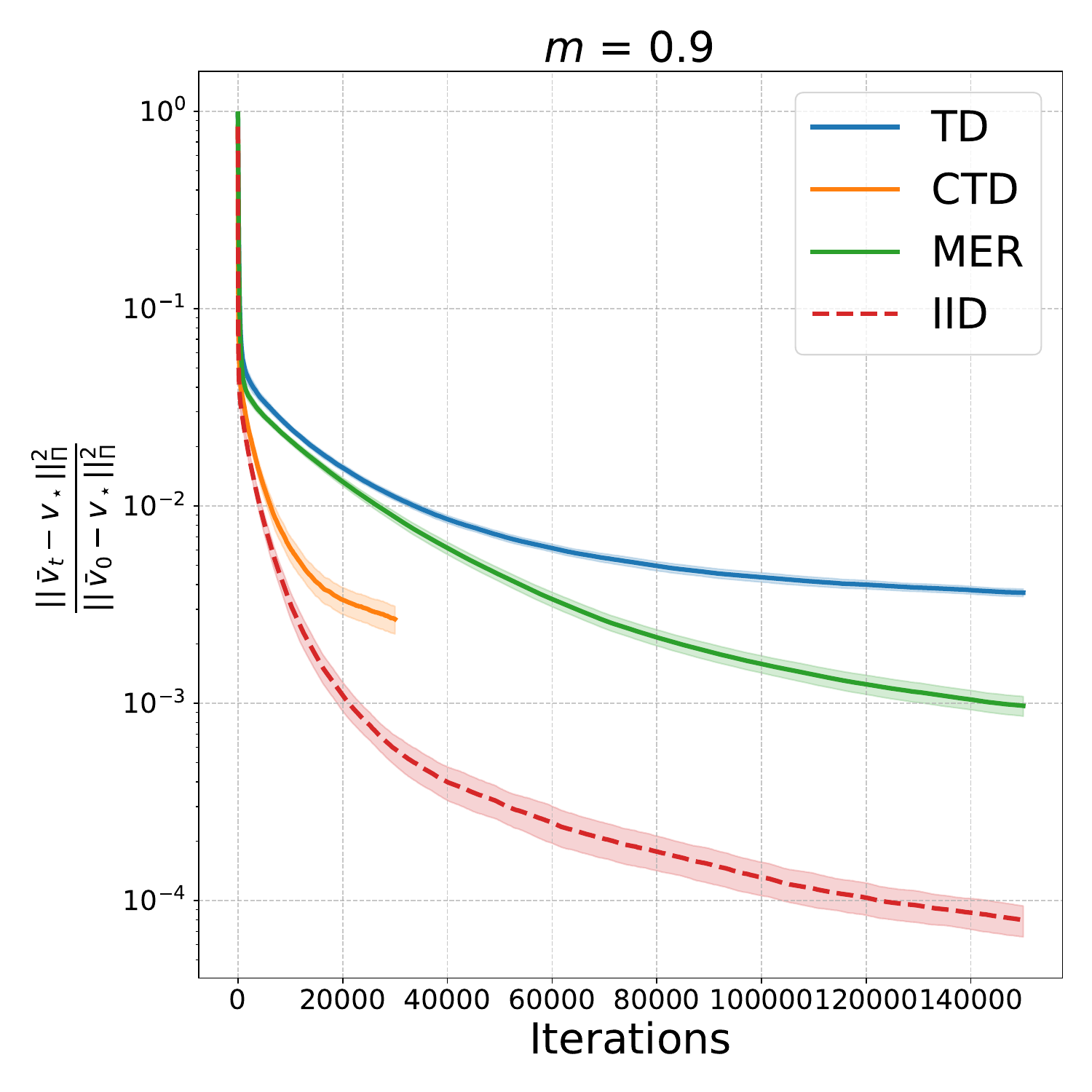}}
    \subfigure[]{\label{fig:2}\includegraphics[width=8cm]{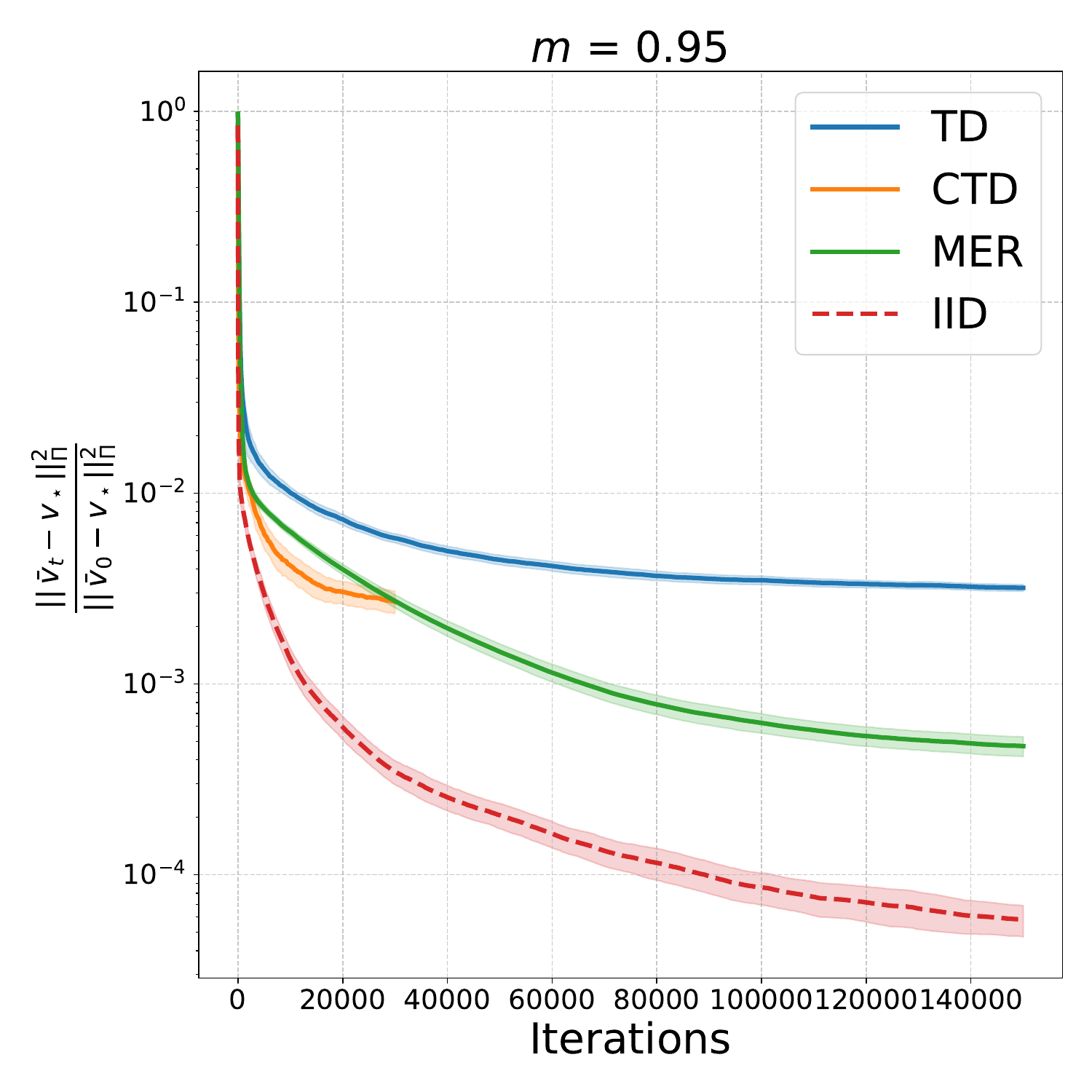}}
    \caption{Comparison of the algorithms for the MDP example. On the $y$-axis we plot the ratios in the Euclidean norm $\|\cdot\|_{\Pi}$ averaged over $50$ instances.}
\end{figure}

To generate the i.i.d.\ curves, we sample independently from the stationary distribution at each time step and apply the TD algorithm with this sample. For the MER, TD, and CTD algorithms, results are obtained using a static Markovian buffer of size $150000$. The skipping parameter for CTD is set to $5$ in all experiments. With our buffer size, the CTD algorithm thus utilizes $30000$ samples in each case due to skipping.

For $m = 0.9$ (Fig.~\ref{fig:1}), the chain mixes rapidly. Owing to skipping, the CTD algorithm closely emulates an i.i.d.\ data stream over the $30000$ iterations over which it is run, but the error cannot be driven further down because we have exhausted all available samples. The TD algorithm achieves comparable error eventually, but the decay of its error over iterations is noticeably worse. The MER algorithm achieves significantly better iterate-by-iterate performance than TD, while plateauing at a smaller eventual error than CTD and TD (and closer to the i.i.d. error). 
When $m$ increases to $0.95$ (Fig.~\ref{fig:2}), the chain mixes more slowly and Markovian dependence becomes significant. Here, the fixed skipping parameter of $5$ for CTD is much smaller than the effective mixing time, which leads to a marked drop in CTD performance. The MER algorithm still performs similarly to the case when $m = 0.9$, showing that its performance is not as sensitive to the mixing parameter. The TD algorithm remains the least effective in this regime.

In summary, CTD's performance is highly sensitive to the choice of skipping parameter, while MER consistently approximates i.i.d.\ performance regardless of the mixing time. As expected, TD performs the worst for both values of $m$.

\subsection{Logistic regression with autoregressive covariates}

We now test the algorithms on a logistic regression task.  The link function for the GLM is taken to be the sigmoid:
\[
    f: z \mapsto \frac{\exp(z)}{1 + \exp(z)}.
\]
The covariates $(a_t)_{t \geq 1}$ are generated using an autoregressive process 
\begin{align}\label{eq:autoregressive}
    a_{t+1} = A a_{t} + \epsilon_{t}, \text{ for }t=1, 2\ldots,
\end{align}
where $\epsilon_t \in \mathbb{R}^n$ is a sequence of i.i.d. Gaussian random variables with mean $0$ and variance $10^{-2}$. The first covariate $a_1$ is obtained by initializing the process at $0$ and using a burn-in period of $10000$ steps.

The matrix $A$ is generated to be a stable, symmetric matrix. 
The degree of Markovian dependence in the covariates is controlled using the number of eigenvalues of $A$ that are close to $1$, which we denote as $d_{\textsc{L}}$. Larger $d_{\textsc{L}}$ values correspond to slower mixing.
Specifically, given $d_{\textsc{L}}$, we generate $A$ via the following steps:
\begin{enumerate}
    \item Eigenvalue generation: We generate a vector of eigenvalues where the first $d_{\textsc{L}}$ entries are set equal to $0.995$, and the remaining eigenvalues are random values chosen from a Gaussian distribution with mean $0$ and variance $10^{-4}$. This ensures that the remaining $d - d_{\textsc{L}}$ eigenvalues are close to $0$. All eigenvalues are collected in a $d \times d$ diagonal matrix $E$.
    \item Matrix Construction:  We now generate uniformly random $d \times d$ orthonormal matrix $Q$, and $A$ is obtained via the eigenvalue decomposition
    \begin{align*}
        A = QEQ^T.
    \end{align*}
\end{enumerate}

After generating the matrix $A$ the covariates $(a_t)_{t \geq 1}$ are generated according to the autoregressive process described in Eq.~\eqref{eq:autoregressive} and the corresponding observations $(y_t)_{t \geq 1}$ are generated according to the GLM in Eq.~\eqref{eq:GLM_model} using the sigmoid link function. 

We fix the covariate dimension at $d=1000$ and the number of samples as $n=10000$ and report results for $d_{\textsc{L}} = 400$ and $500$. The skipping parameter of the CTD algorithm is set as $10$ across all experiments. Thus, for a fixed buffer size of $10000$ samples the CTD algorithm utilizes only $1000$ samples due to skipping. For each value of $d_{\textsc{L}}$, we plot the normalized squared error $\frac{\|x_{T_k} - x_*\|^2}{\|x_0 - x_*\|^2}$, where $x_{T_k}$ is the output, $x_*$ is the optimal solution, and $x_0$ is the initialization.

\begin{figure}[htbp]
    \centering
    \subfigure[]{\label{fig:3}\includegraphics[width=8cm]{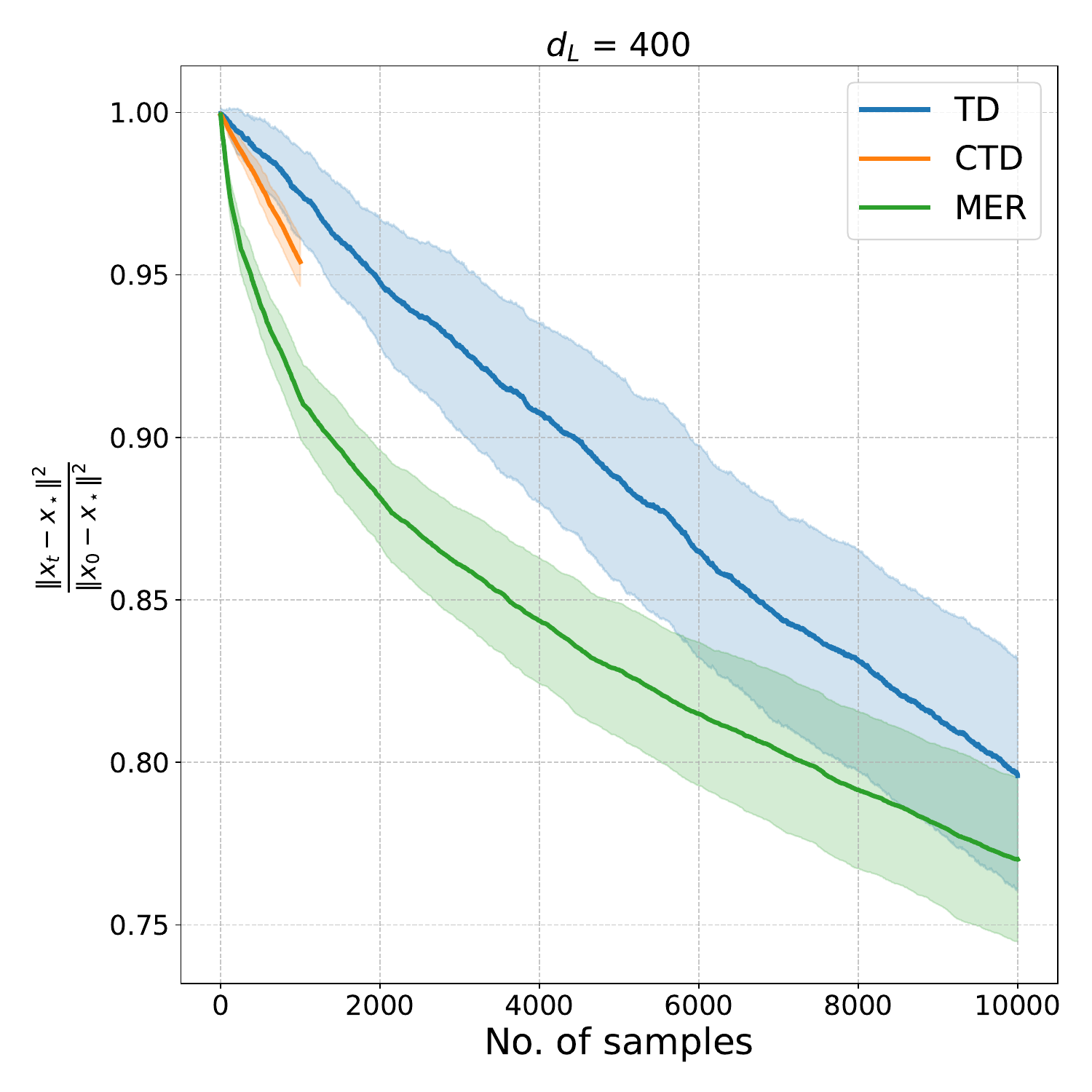}}
    \subfigure[]{\label{fig:4}\includegraphics[width=8cm]{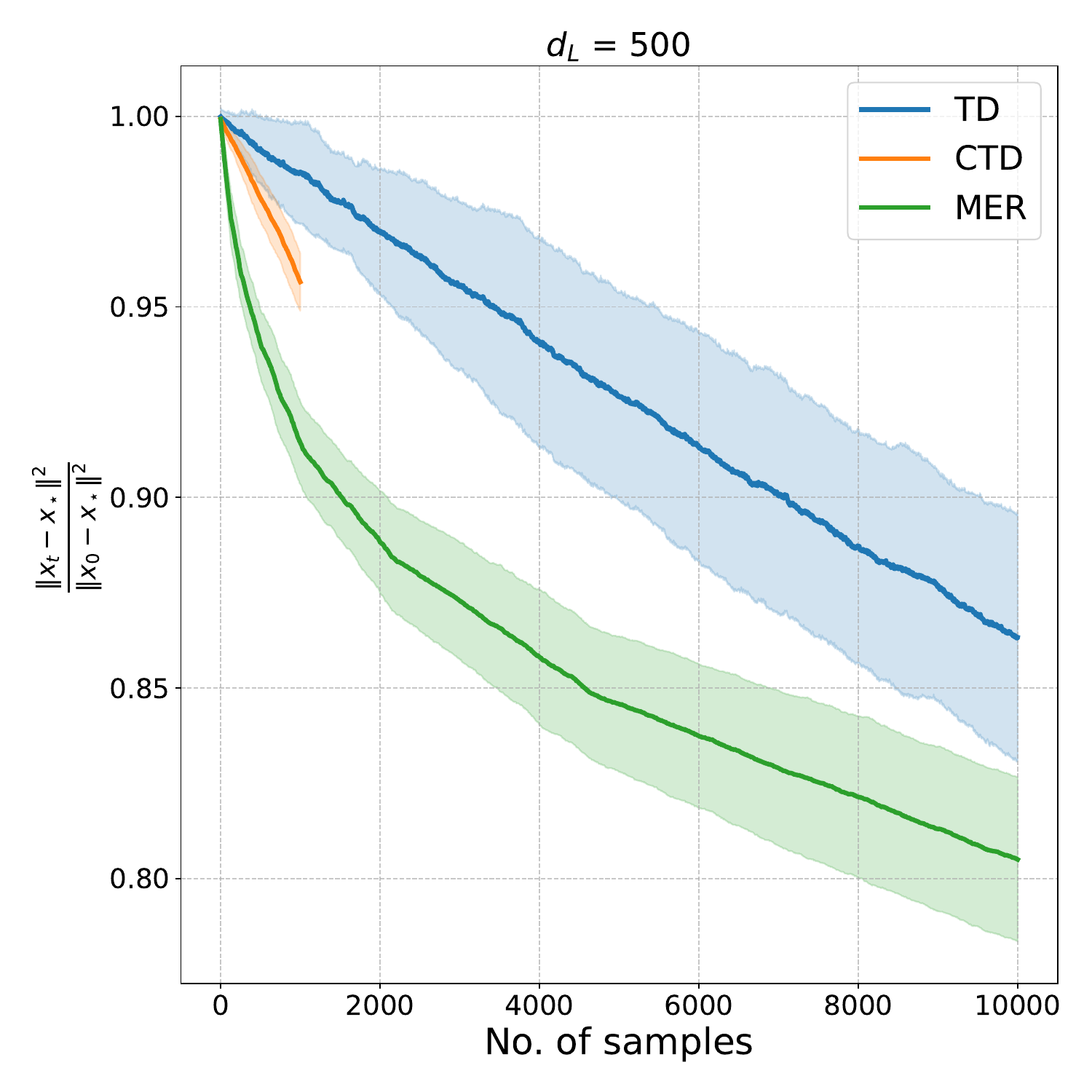}}
    \caption{Comparison of the algorithms for the logistic regression example. On the $y$-axis we plot the ratios in the Euclidean norm $\|\cdot\|_{2}$ averaged over $50$ instances.}
\end{figure}

Figures~\ref{fig:3} and~\ref{fig:4} show that the MER algorithm consistently achieves performance better than TD and CTD, even when the chain mixes slowly. As before, MER achieves this performance without relying on any mixing information.

\section{Proofs}
\label{sec:proofs}

In this section, we present proofs of all our main results. 

We note that while the MER algorithm utilizes a dynamic buffer where samples are deleted and appended (as seen in Algorithm~\ref{alg:two}), the theoretical convergence analysis for a static buffer case extends directly to this setting. The key justification lies in the \emph{epoch-wise re-initialization} of the iterate $x_1^{(k)}$. At the start of each epoch $k$, the variable is initialized independently of its previous trajectory. This mechanism ensures that for the duration of the $k$-th epoch, the algorithm effectively operates on a fixed snapshot of the memory buffer. In the dynamic buffer, the samples used are even further apart, thus the performance guarantees derived under the static buffer assumption hold for each epoch individually. Consequently, the convergence guarantees established herein for the static case remain globally valid for the dynamic buffer implementation of Algorithm~\ref{alg:two}.

\subsection{Preliminary supporting lemmas}

We begin by  stating and proving some lemmas that are useful in multiple parts of the analysis.
\begin{lemma}
    \label{lemma:mixing_norm}
    Suppose Assumptions \ref{assmp:optimal_condition} and \ref{assmp:mixing_norm_revised} hold. 
    For every $t,t' \in \mathbb{Z}_+$ and $x \in X$, with probability $1$,
    \begin{align}\label{eq:mixing_norm}
    \|F(x)-\EE[\widetilde{F}(x,\xi_{t+t'})|\Fspace_{t-1}]\| \leq \cm \cdot \rho^{t'} +\cp \cdot \rho^{t'}\|x-x^{*}\|.
\end{align}
Consequently, with probability 1,
\begin{align}\label{eq:mixing_inner_prod}
    \left\langle F(x)-\EE[\widetilde{F}(x,\xi_{t+t'})|\Fspace_{t-1}], x-x^*  \right\rangle\leq \cm \cdot \rho^{t'}\|x-x^*\| + \cp \cdot \rho^{t'}\|x-x^*\|^2.
\end{align}
\end{lemma}
\begin{proof}
    By setting $y= x^*$ in Ineq.~\eqref{eq:mixing_norm_revised}, we obtain
    \begin{align*}
        \cp \cdot \rho^{t'} \|x - x^*\| &\geq \|F(x)-\EE[\widetilde{F}(x,\xi_{t+t'})|\Fspace_{t-1}]-F(x^*)+\EE[\widetilde{F}(x^*,\xi_{t+t'})|\Fspace_{t-1}]\| \\
        &\geq \|F(x)-\EE[\widetilde{F}(x,\xi_{t+t'})|\Fspace_{t-1}]\|-\|F(x^*)-\EE[\widetilde{F}(x^*,\xi_{t+t'})|\Fspace_{t-1}]\|.
    \end{align*}
    Substituting Ineq.~\eqref{eq:optimal_condition} into the inequality above and rearranging the terms, we have
    \begin{align*}
        \|F(x)-\EE[\widetilde{F}(x,\xi_{t+t'})|\Fspace_{t-1}]\| \leq \cm\rho^{t'}+\cp \rho^{t'} \|x - x^*\|,
    \end{align*}
    which completes the proof of Lemma \ref{lemma:mixing_norm}.
\end{proof}

\begin{lemma}\label{bias_variance_lemma}
    Suppose Assumptions~\ref{assmp:lipscitz1}--\ref{assmp:optimal_condition} hold. For every $t, t' \in \mathbb{Z}_+$ and $x \in X$, we have
    \begin{align*}
        \EE\left[\|\widetilde{F}(x,\xi_{t+t'})-  F(x)\|^2 \right] \leq \sigma^2 +4\cm^2\rho^{2t'}+ (\zeta^2 + 8\bar{L}^2)\EE\left[\|x-x^*\|^2\right],
    \end{align*}
    where $\bar{L} \defn L + \widetilde{L}_1$.
\end{lemma}
\begin{proof}First, by Young's inequality, we have
\begin{align}
\label{eq:intermediate_bias_variance}
    \|\widetilde{F}(x,\xi_{t+t'})-  F(x)\|^2 &{\leq} 2 \|\widetilde{F}(x,\xi_{t+t'})-\EE[\widetilde{F}(x,\xi_{t+t'})|\Fspace_{t-1}]\|^2+ 2 \|\EE[\widetilde{F}(x,\xi_{t+t'})|\Fspace_{t-1}]-  F(x)\|^2.
\end{align}
To bound the second term note, that
\begin{align*}
     \|\EE[\widetilde{F}(x,\xi_{t+t'})|\Fspace_{t-1}]-  F(x)\|^2 & \overset{\1}{\leq} 2\|\EE[\widetilde{F}(x,\xi_{t+t'})|\Fspace_{t-1}]-  F(x) -\EE[\widetilde{F}(x^*,\xi_{t+t'})|\Fspace_{t-1}] + F(x^*) \|^2 \\ \nonumber
     & \quad + 2 \|\EE[\widetilde{F}(x^*,\xi_{t+t'})|\Fspace_{t-1}] - F(x^*)\|^2 \\
     &\overset{\2}{\leq} 4\|F(x^{*})-  F(x)\|^2+ 4\|\EE[\widetilde{F}(x,\xi_{t+t'})|\Fspace_{t-1}]-\EE[\widetilde{F}(x^{*},\xi_{t+t'})|\Fspace_{t-1}]\|^2 \\
     & \quad + 2 \cm^2 \rho^{2t'} \\ 
     &\overset{\3}{\leq} 4 L^2 \|x-x^{*}\|^2 + 4 \EE[\|\widetilde{F}(x,\xi_{t+t'})-\widetilde{F}(x^{*},\xi_{t+t'})\|^2 |\Fspace_{t-1}] + 2 \cm^2 \rho^{2t'}\\
     &\overset{\4}{\leq} 4 L^2 \|x-x^{*}\|^2 + 4 \EE[\widetilde{L}_1^2 \|x-x^{*}\|^2 |\Fspace_{t-1}]+ 2 \cm^2 \rho^{2t'},
\end{align*}
where step $\1$ follows from the Young's inequality, step $\2$ follows from Assumption~\ref{assmp:optimal_condition} and Jensen's inequality, step $\3$ follows from Ineq.~\eqref{eq:F_lipschitz}, and step $\4$ follows from Assumption~\ref{assmp:lipscitz1}.
Taking expectation and using Assumption \ref{assmp:state_dependent_noise} in Ineq.~\eqref{eq:intermediate_bias_variance} we have
\begin{align}
    \EE\left[\|\widetilde{F}(x,\xi_{t+t'})-  F(x)\|^2 \right] \leq \sigma^2 +4\cm^2\rho^{2t'}+ (\zeta^2 + 8\bar{L}^2)\EE\left[\|x-x^*\|^2\right],
\end{align}
where we have further used $L^2 +\widetilde{L}_1^2 \leq (L+\widetilde{L}_1)^2 = \bar{L}^2$. This completes the proof of Lemma~\ref{bias_variance_lemma}.
\end{proof}
\subsection{Proof of Theorem \ref{thm:main_error}}\label{subsec:proof_theorem_1}

We start by defining the shorthand
\begin{align}
    \label{eq:theta_defn}
    \Gamma_t \defn (1+\mu \eta_k)^{t}.
\end{align}
The error decomposition in the following proposition is key to the proof. 
\begin{proposition}\label{recursion_proposition}
 Suppose Assumptions \ref{assmp:optimal_ball}-\ref{assmp:mixing_norm_revised} hold, then for the $k$-th epoch of Algorithm \ref{alg:two} we have
    \begin{align}
    \nonumber
         \EE\|\finaliterate - x^*\|^2&\leq \underbrace{\frac{G_1}{A_{T_k}} \|\initialiterate - x^*\|^2}_{R_1}+ \underbrace{\frac{10 \eta_k \cm}{A_{T_k}}\left(\sum_{t=1}^{\tau_M/\tau_k}\rho^{t\tau_k-1}\Gamma_t + \sum_{t=\tau_M/\tau_k + 1}^{T_k}\rho^{\tau_M+\tau_k-1}\Gamma_t\right)}_{R_2} \\ \nonumber 
    &\qquad+ \underbrace{\frac{H}{A_{T_k}} \eta_k^2\left(6\sigma^2+25\cm^2 \rho^{2(\tau_k-1)}+ 4 \|F(x^*)\|^2\right)}_{R_3},
    \end{align}
   where 
\begin{align}
\nonumber
    A_{T_k} &= \left(\frac{1}{2} -\eta_k^2 \bar{L}^2 + \eta_k \mu \right)\Gamma_{T_k}, \\ \nonumber
    G_1 &= \left(\frac{1}{2} + 2\eta_k^2E \right)(1+\eta_k \mu)+\eta_kC \sum\limits_{t'=1}^{\alpha_k+1}\rho^{\tau_k t'-1}(1+\eta_k \mu)^{t'} +2\eta_k^2 \bar{L}^2\sum\limits_{t'=1}^{\alpha_k+1} (1+\eta_k \mu)^{t'}\\
    &\quad~  +2\eta_k^2 \bar{L}^2\sum\limits_{t'=1}^{\alpha_k} (t'-1)(1+\eta_k \mu)^{t'}, \nonumber\\ \nonumber
    H &= \sum_{t=1}^{\alpha_k}t\Gamma_t+\sum_{t=\alpha_k+1}^{T_k}(\alpha_k+1)\Gamma_t ,\; E = \zeta^2 + 8\bar{L}^2,
\end{align}
and $\alpha_k$ is defined in Eq. \eqref{eq:alpha_definition}. 
\end{proposition}
In the definition for $G_1$ above, recall that $C = \left(\frac{\cm}{40}+\cp\right)$ and $\bar{L} = L+ \widetilde{L}_1$ from Eq.~\eqref{eq:definitions}.
Proposition~\ref{recursion_proposition} is proved in Section~\ref{sec:pf_prop1}. 
Paralleling the theorem statement, the first term $R_1$ is the deterministic error, the term $R_2$ is due to the bias at the optimal solution $x^{*}$ and the term $R_3$ is the stochastic error. Note that $R_2=0$ if $\cm=0$.

For the time being, we assume the statement of Proposition \ref{recursion_proposition} as given and proceed with our proof for Theorem \ref{thm:main_error}.    
Next, we state a simple condition and a claim which is used throughout the proof of Theorem \ref{thm:main_error}:
\begin{align}
\label{eq:numerical_bound}
    \frac{3\mu^2}{8\left( \alpha_k+1 \right)(\zeta^2 + 16 \bar{L}^2)} \leq \frac{3\mu^2}{128 \bar{L}^2} \overset{\1}{\leq} \frac{3}{128},
\end{align}
where step $\1$ follows from the fact $\mu \leq \bar{L}$.

\begin{claim}\label{claim:eta}
Under the conditions of Theorem~\ref{thm:main_error}, the step-size schedule $\eta_k$ satisfies
\begin{align} \label{eq:eta_observation}
    \eta_k \leq \min \left\{\frac{3 \mu}{8\left( \alpha_k+1 \right)(\zeta^2 + 16\bar{L}^2)},\frac{p_k\log T_k}{\mu T_k} \right\}.
\end{align}
\end{claim}
We take this statement as given for the moment and prove it in Section~\ref{pf:claim_eta}.

In order to proceed with our proof we need to bound $\Gamma_{\alpha_k+1}$, where $\Gamma_t$ is defined in Eq.~\eqref{eq:theta_defn}. We bound this term as follows:
\begin{align}\label{eq:theta_bound}
    \Gamma_{\alpha_k+1}   = (1+\eta_k \mu)^{\alpha_k+1} &\leq \left(1+ \frac{3\mu^2}{8(\zeta^2+16\bar{L}^2) \left(\alpha_k+1 \right)}\right)^{\alpha_k+1} \leq e^{\frac{3\mu^2}{8(\zeta^2+16 \bar{L}^2) }} < 2.
\end{align}
We can now bound the different error terms from Proposition~\ref{recursion_proposition}.
\subsubsection{\texorpdfstring{Bounding $R_1$}{}}
First we bound $G_1$ for the choice $\Gamma_t = (1+\eta_k \mu)^t$ and $\eta_k$ defined in Theorem~\ref{thm:main_error}. From the definition of $G_1$ in Proposition~\ref{recursion_proposition} we have
\begin{align*}
    G_1 &= \left(\frac{1}{2} + 2\eta_k^2E \right)(1+\eta_k \mu)+\eta_kC \sum\limits_{t'=1}^{\alpha_k+1}\rho^{\tau_k t'-1}(1+\eta_k \mu)^{t'} +2\eta_k^2 \bar{L}^2\sum\limits_{t'=1}^{\alpha_k+1} (1+\eta_k \mu)^{t'}\\
    &\quad~  +2\eta_k^2 \bar{L}^2\sum\limits_{t'=1}^{\alpha_k} (t'-1)(1+\eta_k \mu)^{t'} \\ 
    &\overset{\1}{\leq} \left(\frac{1}{2} + \eta_k^2E \right)(1+\eta_k \mu)+ \eta_k C \sum\limits_{t'=1}^{\alpha_k+1}(1+\eta_k \mu)^{t'} +2\eta_k^2 \bar{L}^2\frac{(1+ \eta_k \mu)^{\alpha_k+2}}{\mu \eta_k} + 2\eta_k^2 \bar{L}^2  \alpha_k\frac{(1+ \eta_k \mu)^{\alpha_k+1}}{\mu \eta_k} \\ 
    &\overset{}{\leq} \left(\frac{1}{2} + \eta_k^2E \right)(1+\eta_k \mu)+  C \frac{(1+ \eta_k \mu)^{\alpha_k+1}}{\mu} + 2\eta_k \bar{L}^2 \left(\alpha_k+(1+\eta_k \mu)\right)\frac{(1+ \eta_k \mu)^{\alpha_k+1}}{\mu },
\end{align*}
where step $\1$ follows from fact that $\rho$ lies in the interval $(0,1)$. Substituting the definitions for $E$ and condition for $\eta_k$ into the inequality above, we have
\begin{align*}
    G_1&\overset{\1}{\leq} \left(\frac{1}{2} + \frac{9 \mu^2}{\left( \alpha_k+1 \right)^2 64(\zeta^2 + 16 \bar{L}^2)^2}(\zeta^2+8 \bar{L}^2) \right) \left( 1 + \frac{3 \mu^2}{8\left( \alpha_k+1 \right)(\zeta^2 + 16 \bar{L}^2)}\right)\\ \nonumber
    &\qquad \qquad +  \frac{2C}{\mu} + \frac{6 \mu}{8\left( \alpha_k+1 \right)(\zeta^2 + 16 \bar{L}^2)} \bar{L}^2 \left(\alpha_k+1+\frac{3 \mu^2}{8\left( \alpha_k+1 \right)(\zeta^2 + 16 \bar{L}^2)}\right) \frac{2}{\mu} \\ 
    &\overset{\2}{\leq} \left(\frac{1}{2} + \frac{9 }{1024} \right)\frac{131}{128}  +   \frac{2C}{\mu} + \frac{12}{128}\frac{\alpha_k + 131/128 }{\alpha_k+1}  \leq 1+  \frac{2C}{\mu},
\end{align*}
where step $\1$ follows from Ineq.~\eqref{eq:eta_observation} and Ineq.~\eqref{eq:theta_bound} and step $\2$ follows from Ineq.~\eqref{eq:numerical_bound}.

Next we bound $A_{T_k}$. From the definition of $A_{T_k}$ in Proposition~\ref{recursion_proposition} we have
\begin{align}
    \nonumber
    A_{T_k}&=\left(\frac{1}{2} -\eta_k^2 \bar{L}^2 + \eta_k \mu \right)\Gamma_{T_k} = \left(\frac{1}{2} -\eta_k^2 \bar{L}^2 + \eta_k \mu  \right)(1+\eta_k \mu)^{T_k} \\ \nonumber
    &\geq \left(\frac{1}{2} -\eta_k \mu  \frac{3 \bar{L}^2}{8\left( \alpha_k+1 \right)(\zeta^2 + 16 \bar{L}^2)} + \eta_k \mu  \right)(1+\eta_k \mu)^{T_k} \\ \label{eq:A_bound}
    &\geq \left(\frac{1}{2} -  \frac{3 \eta_k \mu}{128} + \eta_k \mu  \right)(1+\eta_k \mu)^{T_k} = \left(\frac{1}{2} +  \frac{125 \eta_k \mu}{128}  \right)(1+\eta_k \mu)^{T_k} \geq \frac{1}{2}(1+\eta_k \mu)^{T_k}.
\end{align}
Finally combining our bounds on $G_1$ and $A_{T_k}$, we obtain
\begin{align*}
    \frac{G_1}{A_{T_k}}\leq 2\left(1+\frac{2C}{\mu} \right)\left(1+\eta_k \mu\right)^{-T_k}.
\end{align*}
Using the step-size upper bound from Ineq.~\eqref{eq:eta_observation} and replacing the max of two positive quantities by their sum, we obtain
\begin{align*}
    2&\left(1+\frac{2C}{\mu} \right)\left(1+\eta_k \mu\right)^{-T_k}\\
    &\leq 2\left(1+\frac{2C}{\mu} \right)\left(1+\frac{3 \mu^2}{\left( \alpha_k+1 \right)8(\zeta^2 + 16 \bar{L}^2)} \right)^{-T_k} + 2 \left(1+\frac{2C}{\mu} \right)\left(1+\frac{p_k \log T_k}{ T_k} \right)^{-T_k} \\ 
    &\leq 2\left(1+\frac{2C}{\mu} \right)\left(1+\frac{3 \mu^2}{\left( \alpha_k+1 \right)8(\zeta^2 + 16 \bar{L}^2)} \right)^{-T_k} + 2\left(1+\frac{2C}{\mu} \right) T_{k}^{-p_k} \\ \nonumber
    &\overset{\1}{=} 2\left(1+\frac{2C}{\mu} \right)\left(1+\frac{3 \mu^2}{\left( \alpha_k+1 \right)8(\zeta^2 + 16 \bar{L}^2)} \right)^{-T_k} + \frac{6\left(1+\frac{2C}{\mu} \right)}{\mu^2 T_k M \max(1,D^2)}   (6\sigma^2+ 4 \|F(x^*)\|^2),
\end{align*}
where step $\1$ follows by substituting the definition of $p_k$ from Eq.~\eqref{eq:definitions}. Continuing we have
\begin{align*}
    2\left(1+\frac{2C}{\mu} \right)\left(1+\eta_k \mu\right)^{-T_k} &\overset{\1}{\leq} 2\left(1+\frac{2C}{\mu} \right)\left(1+\frac{3 \mu^2}{\left( \alpha_k+1 \right)8(\zeta^2 + 16 \bar{L}^2)} \right)^{-T_k} \\ \nonumber
    &\qquad+ \frac{3}{\mu^2 T_k  D^2}   (6\sigma^2+ 4 \|F(x^*)\|^2) \\ \nonumber
    &\leq  2\left(1+\frac{2C}{\mu} \right)\left(1+\frac{3 \mu^2}{\left( \alpha_k+1 \right)8(\zeta^2 + 16 \bar{L}^2)} \right)^{-T_k} \\ \nonumber
    &\qquad+ \frac{3}{\mu^2 T_k  D^2}   (6\sigma^2+25\cm^2 \rho^{2(\tau_k-1)}+ 4 \|F(x^*)\|^2).
\end{align*}
where step $\1$ follows from the definition of $M$ in Eq.~\eqref{eq:definitions} and using the fact that $\frac{1}{\max(a,b)}\leq \frac{1}{b}$.
Thus we obtain
\begin{align}
\nonumber
    R_1 = \frac{G_1}{A_{T_k}} \|\initialiterate - x^*\|^2 &\leq 2\left(1+\frac{2C}{\mu} \right)\left(1+\frac{3 \mu^2}{\left( \alpha_k+1 \right)8(\zeta^2 + 16 \bar{L}^2)} \right)^{-T_k}\|\initialiterate - x^*\|^2 \\ \nonumber
    & \qquad+ \frac{3}{\mu^2 T_k D^2}   (6\sigma^2+25\cm^2 \rho^{2(\tau_k-1)}+ 4 \|F(x^*)\|^2)\|\initialiterate - x^*\|^2.
\end{align}
Substituting $\frac{\|\initialiterate - x^*\|}{D} \leq 1$ from Assumption~\ref{assmp:optimal_ball} we obtain
\begin{align} \nonumber
    R_1 &\leq 2\left(1+\frac{2C}{\mu} \right)\left(1+\frac{3 \mu^2}{\left( \alpha_k+1 \right)8(\zeta^2 + 16 \bar{L}^2)} \right)^{-T_k}D^2 \\ \label{eq:deterministic_error}
    & \qquad+ \frac{3}{\mu^2 T_k}   (6\sigma^2+25\cm^2 \rho^{2(\tau_k-1)}+ 4 \|F(x^*)\|^2).
\end{align}
\subsubsection{\texorpdfstring{Bounding $R_2$}{}}
Now consider the error due to the bias at the optimal solution $x^*$. We have
\begin{align*}
    R_2 &= \frac{10 \eta_k \cm}{A_{T_k}}\left(\sum_{t=1}^{\tau_M/\tau_k}\rho^{t\tau_k-1}(1+\eta_k \mu)^t + \sum_{t=\tau_M/\tau_k + 1}^{T_k}\rho^{\tau_M+\tau_k-1}(1+\eta_k \mu)^t\right) \\ 
    &\overset{\1}{\leq} \frac{10 \eta_k \cm}{A_{T_k}}\left( \sum_{t=1}^{\tau_M/\tau_k}(1+\eta_k \mu)^t +  \frac{\rho^{\tau_M+\tau_k-1}(1+\eta_k \mu)^{T_k+1}}{\eta_k \mu}\right) \\ 
    &\overset{\2}{\leq} \frac{40 \cm}{\mu}(1+\eta_k \mu)^{-T_k} +  \frac{20  \cm\rho^{\tau_M+\tau_k-1}}{ \mu}(1+\eta_k \mu),
\end{align*}
where step $\1$ follows from the fact that $\rho \in (0,1)$ and step $\2$ follows from Ineq. \eqref{eq:theta_bound} and the bound on $A_{T_k}$ from Ineq. \eqref{eq:A_bound}. We bound the first term in the above inequality similarly to how we bounded the ratio $G_1/A_{T_k}$ in $R_1$. In particular, using Ineq.~\eqref{eq:eta_observation} and replacing the max of two positive quantities by their sum, we obtain
\begin{align*}
    \left(1+\eta_k \mu\right)^{-T_k} &\leq \left(1+\frac{3 \mu^2}{\left( \alpha_k+1 \right)8(\zeta^2 + 16 \bar{L}^2)} \right)^{-T_k} + \left(1+\frac{p_k \log T_k}{ T_k} \right)^{-T_k} \\ 
    &\leq \left(1+\frac{3 \mu^2}{\left( \alpha_k+1 \right)8(\zeta^2 + 16 \bar{L}^2)} \right)^{-T_k} +  T_{k}^{-p_k} \\ 
    &\overset{}{=} \left(1+\frac{3 \mu^2}{\left( \alpha_k+1 \right)8(\zeta^2 + 16 \bar{L}^2)} \right)^{-T_k} \\ 
    &\qquad + \frac{3}{\mu^2 M T_k\max (1,\|\initialiterate   -x^*\|^2)}   (6\sigma^2+ 4 \|F(x^*)\|^2) \\
    &\overset{}{\leq} \left(1+\frac{3 \mu^2}{\left( \alpha_k+1 \right)8(\zeta^2 + 16 \bar{L}^2)} \right)^{-T_k}
     + \frac{3}{\mu^2 \frac{40\cm}{\mu} T_k}   (6\sigma^2+25\cm^2 \rho^{2(\tau_k-1)}+ 4 \|F(x^*)\|^2).
\end{align*}
Thus the error due to bias at $x^*$ is bounded as
\begin{align}
\nonumber
    R_2 &\leq \frac{40 \cm}{\mu} \left(1+\frac{3 \mu^2}{\left( \alpha_k+1 \right)8(\zeta^2 + 16\bar{L}^2)} \right)^{-T_k} + \frac{3}{\mu^2  T_k}   (6\sigma^2+25\cm^2 \rho^{2(\tau_k-1)}+ 4 \|F(x^*)\|^2) \\ \label{eq:optimal_bias_error}
    &\qquad +\frac{20  \cm\rho^{\tau_M+\tau_k-1}}{ \mu}(1+\eta_k \mu).
\end{align}
\subsubsection{\texorpdfstring{Bounding $R_3$}{}}
We now bound the stochastic error term. Substituting $\Gamma_t = (1+\eta_k \mu)^t$ in the definition of $D$ from Proposition \ref{recursion_proposition}, we obtain
\begin{align*}
    H &= \sum_{t=1}^{\tau_M/\tau_k}t(1+\eta_k \mu)^t+\sum_{t=\tau_M/\tau_k+1}^{T_k}(\tau_M/\tau_k+1)(1+\eta_k \mu)^t \\ 
    &\leq\sum_{t=1}^{T_k}(\tau_M/\tau_k+1)(1+\eta_k \mu)^t \\ \nonumber
    &\leq \left(\alpha_k +1 \right)\frac{(1+\mu \eta_k)^{T_k +1}}{\mu \eta_k}.
\end{align*}
Combining the bounds for $H$ from above and $A_{T_k}$ from Ineq.~\eqref{eq:A_bound}, we obtain
\begin{align*}
    \eta_k^2\frac{H}{A_{T_k}} &\leq 2\eta_k \left(\alpha_k +1 \right)\frac{(1+\mu \eta_k)}{\mu }\\
    &\overset{\1}{\leq} 2\eta_k \left(\alpha_k +1 \right)\frac{1+\frac{3 \mu^2}{8\left( \alpha_k+1 \right)(\zeta^2 + 16 \bar{L}^2)}}{\mu} \\
    &\overset{\2}{\leq} \frac{3 \eta_k}{\mu} \left(\alpha_k +1 \right) \leq  \frac{3p_k \log T_k}{\mu^2 T_k}\left(\alpha_k +1 \right),
\end{align*}
where step $\1$ follows from Ineq.~\eqref{eq:eta_observation} and step $\2$ follows from Ineq.~\eqref{eq:numerical_bound}.
Thus, 
\begin{align}
    \nonumber
     R_3 &= \frac{H}{A_{T_k}} \eta_k^2(6\sigma^2+25\cm^2 \rho^{2(\tau_k-1)}+ 4 \|F(x^*)\|^2) \\ \label{eq:stochastic_error} 
     &\leq  \frac{3p_k \log T_k}{\mu^2 T_k} \left(\alpha_k +1 \right)  (6\sigma^2+25\cm^2 \rho^{2\tau_k}+ 4 \|F(x^*)\|^2).
\end{align}

\subsubsection{Combining the bounds}
Substituting bounds from Ineqs.~\eqref{eq:deterministic_error}, \eqref{eq:optimal_bias_error}, and \eqref{eq:stochastic_error} into Proposition \ref{recursion_proposition}, we have that $\EE\|\finaliterate - x^*\|^2$ is bounded above by
\begin{align}
\nonumber
    &\frac{G_1}{A_{T_k}} \|\initialiterate - x^*\|^2+ \frac{H}{A_{T_k}} \eta_k^2(6\sigma^2+25\cm^2 \rho^{2(\tau_k-1)}+ 4 \|F(x^*)\|^2)+ \frac{1}{A_{T_k}}\sum_{t=1}^{T_k} \Gamma_t\frac{10\eta_k \cm^2 \rho^{(\kappa+1) \tau_k-1}}{\mu} \\ \nonumber &\leq 2\left(1 + \frac{2C}{\mu} \right)\left(1+\frac{3 \mu^2}{\left( \alpha_k+1 \right)8(\zeta^2 + 16 \bar{L}^2)} \right)^{-T_k}D^2 
    \\ \nonumber
    & \quad + \left(\frac{3}{\mu^2 T_k} + \frac{3p_k \log T_k}{\mu^2 T_k} \left(\alpha_k +1 \right)+\frac{3}{\mu^2 T_k}\right)  (6\sigma^2+25\cm^2 \rho^{2(\tau_k-1)}+ 4 \|F(x^*)\|^2) 
    \\ \nonumber 
    &\qquad +\frac{40 \cm}{\mu} \left(1+\frac{3 \mu^2}{\left( \alpha_k+1 \right)8(\zeta^2 + 16 \bar{L}^2)} \right)^{-T_k} +\frac{20  \cm\rho^{\tau_M+\tau_k-1}}{ \mu}(1+\eta_k \mu) \\ \nonumber
    &\leq M \left(1+\frac{3 \mu^2}{\left( \alpha_k+1 \right)8(\zeta^2 + 16 \bar{L}^2)} \right)^{-T_k}\left(D^2+1  \right)\\ \nonumber &\qquad + \frac{3(2+p_k\log T_k)}{\mu^2 T_k} \left(\alpha_k +1 \right)  (6\sigma^2+25\cm^2 \rho^{2(\tau_k-1)}+ 4 \|F(x^*)\|^2) + \frac{20  \cm\rho^{\tau_M+\tau_k-1}}{ \mu}(1+\eta_k \mu).
    \end{align}
This completes the proof of the arguments in Theorem~\ref{thm:main_error}. It remains to establish Claim~\ref{claim:eta}.
\subsubsection{Proof of Claim~\ref{claim:eta}}\label{pf:claim_eta}

Recall from Theorem~\ref{thm:main_error} that we chose the step-size
\begin{align*}
\eta_k =
\min \left\{\frac{3 \mu}{16(\zeta^2 + 16\bar{L}^2)},\frac{p_k\log T_k}{\mu T_k} \right\}.
\end{align*}
 We now divide the proof of the claim into two cases. 

\paragraph{Case I:}
 For the case $\tau_k \geq \tau_M$, i.e., $\alpha_k \leq 1$ and $\frac{1}{1+\alpha_k} \geq \frac{1}{2}$, we have the following:
\begin{align*}
    \frac{3 \mu}{16(\zeta^2 + 16 \bar{L}^2)} \leq \frac{3 \mu}{8\left( 1+\alpha_k \right)(\zeta^2 + 16 \bar{L}^2)}.
\end{align*}
Thus when $\eta_k = \frac{3 \mu}{16(\zeta^2 + 16 \bar{L}^2)}$ we have
\begin{align*}
    \eta_k = \frac{3 \mu}{16(\zeta^2 + 16 \bar{L}^2)} \leq \frac{3 \mu}{8\left( 1+\alpha_k \right)(\zeta^2 + 16 \bar{L}^2)}. 
\end{align*}
Whenever $\eta_k = \frac{p_k\log T_k}{\mu T_k}$ we have the following set of inequalities:
\begin{align*}
    \eta_k = \frac{p_k\log T_k}{\mu T_k} < \frac{3 \mu}{16(\zeta^2 + 16 \bar{L}^2)} \leq \frac{3 \mu}{8\left( 1+\alpha_k \right)(\zeta^2 + 16 \bar{L}^2)}.
\end{align*}
Thus we obtain
\begin{align*}
    \eta_k \leq \min \left\{\frac{3 \mu}{8(1+\alpha_k)(\zeta^2 + 16\bar{L}^2)},\frac{p_k\log T_k}{\mu T_k} \right\}
\end{align*}
This concludes our proof of Claim~\ref{claim:eta} in Case I.

\paragraph{Case II:} For the case $\tau_k \leq \tau_M$, i.e., $\alpha_k \geq 1$ and $\frac{1}{1+\alpha_k} \leq \frac{1}{2}$, we have
\begin{align}
\label{eq:intermediate_inequality}
    \frac{3 \mu}{16(\zeta^2 + 16 \bar{L}^2)} \geq \frac{3 \mu}{8\left( 1+\alpha_k \right)(\zeta^2 + 16 \bar{L}^2)}.
\end{align}
From the condition $B \geq \frac{32 p_1 \tau_M(\zeta^2 + 16 \bar{L}^2 )}{3\mu^2}\log \left( \frac{32 p_1 \tau_M(\zeta^2 + 16 \bar{L}^2)}{3\mu^2} \right)$ we have the following:
\begin{align*}
    \frac{\log B}{B} \leq \frac{3\mu^2}{16p_1\tau_M (\zeta^2 + 16 \bar{L}^2)}.
\end{align*}
Note that the total number of epochs  is upper bounded by $\log_2 B$, i.e., $k \leq \log_2 B$ for all $k= 1, \ldots, K$. Thus we have 
\begin{align*}
    \frac{p_1 k\log 2}{\mu } \leq  \frac{3 \mu B}{16\tau_M (\zeta^2 + 16 \bar{L}^2)}.
\end{align*}
Substituting the definitions $T_k=2^k$ and $\tau_k = \frac{B}{T_k}$ we have
\begin{align*}
    \frac{p_1 \log T_k}{\mu T_k} \leq  \frac{3 \mu \tau_k}{16\tau_M (\zeta^2 + 16 \bar{L}^2)}.
\end{align*}
Note that $p_k\leq p_1$ since $p_k$ is monotone across the epochs and $\left( 1+\alpha_k\right) \leq  \frac{2\tau_M}{\tau_k}$ in Case II. Substituting these bounds we obtain
\begin{align*}
    \frac{p_k \log T_k}{\mu T_k} \leq \frac{3 \mu}{8\left( 1+\alpha_k \right)(\zeta^2 + 16 \bar{L}^2)}.
\end{align*}
Combining this with Ineq.~\eqref{eq:intermediate_inequality} we obtain the following in Case II:
\begin{align*}
    \eta_k =\frac{p_k \log T_k}{\mu T_k} \leq \frac{3 \mu}{8\left( 1+\alpha_k \right)(\zeta^2 + 16 \bar{L}^2)} \leq \frac{3 \mu}{16(\zeta^2 + 16 \bar{L}^2)}.
\end{align*}

Combining the inequalities for the step-size parameter from Case I and Case II, we observe that the following condition is always satisfied:
\begin{align} 
    \eta_k \leq \min \left\{\frac{3 \mu}{8\left( \alpha_k+1 \right)(\zeta^2 + 16\bar{L}^2)},\frac{p_k\log T_k}{\mu T_k} \right\}.
\end{align}
This concludes the proof of Claim~\ref{claim:eta} and thus of Theorem~\ref{thm:main_error}.
\subsection{Proof of Proposition \ref{recursion_proposition}}\label{sec:pf_prop1}

We start by stating the ``three-point lemma'' for Algorithm~\ref{alg:two}, capturing the optimality condition of the SA update in Eq.~\eqref{eq:MER_update}.
\begin{lemma}
    \label{lemma:three_point_MER}
    Let $\updatediterate$ be generated according to the update defined in Eq.~\eqref{eq:MER_update}. 
    Then for any $x\in X$ we have
    \begin{align}\label{eq:base_eq}
        \eta_k \langle \widetilde{F}(\iterate,\xi_{t\tau_k}), \updatediterate-x \rangle + \frac{1}{2} \|\updatediterate - \iterate\|^2 \leq \frac{1}{2}\|\iterate-x\|^2 - \frac{1}{2}\|\updatediterate-x\|^2.
    \end{align}
\end{lemma}
We prove this lemma later in Section~\ref{sec:pf_lemma_three_point_MER}.

This lemma forms the basis of our analysis. We start by decomposing the inner product in Ineq.~\eqref{eq:base_eq} as follows:
\begin{align}
\nonumber
    \left \langle \widetilde{F}(\iterate,\xi_{t\tau_k}), \updatediterate-x \right\rangle &= \left\langle F(\updatediterate), \updatediterate-x \right\rangle + \left\langle F(\iterate)-F(\updatediterate), \updatediterate-x \right\rangle \\ \nonumber
    &\qquad + \left\langle  \widetilde{F}(\iterate,\xi_{t\tau_k})-  F(\iterate), \updatediterate-x \right\rangle \\ \nonumber
    &\overset{\1}{\geq} \left \langle F(\updatediterate), \updatediterate-x \right \rangle - L \|\updatediterate-\iterate\|\|\updatediterate-x\|\\ \label{eq:intermediate_inner_product_bound}
    &\qquad + \left \langle  \widetilde{F}(\iterate,\xi_{t\tau_k})-  F(\iterate), \updatediterate-x \right \rangle ,
\end{align}
where step $\1$ follows from Cauchy--Schwarz inequality and the Lipschitz property in Ineq. \eqref{eq:F_lipschitz}. We provide the following claim to bound the last term in Ineq.~\eqref{eq:intermediate_inner_product_bound}.

\begin{claim}\label{claim:inner_prod}
    Let $\pastiterate$ be a past iterate for some parameter $\kappa \in \{0,1,2,\ldots\}$. Then,
    \begin{align*}
        \left \langle  \widetilde{F}(\iterate,\xi_{t\tau_k})-  F(\iterate), \updatediterate-x \right \rangle &\geq -\|\widetilde{F}(\iterate,\xi_{t\tau_k})-  F(\iterate)\|\|\updatediterate-\iterate\| \\ 
        &\qquad+ \left\langle  \widetilde{F}(\pastiterate,\xi_{t\tau_k})- F(\pastiterate), \pastiterate-x \right\rangle \\
    &\qquad-\sum_{t' = t-\kappa}^{t-1} \eta_k \bigg(2\bar{L}^2\|x_{t'}^{(k)}-x\|^2+ 2\|\widetilde{F}(x,\xi_{t'\tau_k})\|^2 +2\bar{L}^2\|\pastiterate-x\|^2 \\ \nonumber
    &\qquad \qquad \qquad \qquad+ 2\|\widetilde{F}(\iterate,\xi_{t\tau_k})-  F(\iterate)\|^2 \bigg).
    \end{align*}
\end{claim}
Claim~\ref{claim:inner_prod} is proved later in Section~\ref{sec:pf_claim_inner_prod}.

Substituting Claim~\ref{claim:inner_prod} into Ineq.~\eqref{eq:intermediate_inner_product_bound} we obtain
\begin{align}
    \nonumber
    \big \langle & \widetilde{F}(\iterate,\xi_{t\tau_k}), \updatediterate-x \big\rangle \\ \nonumber
    &\geq \left \langle F(\updatediterate), \updatediterate-x \right \rangle - L \|\updatediterate-\iterate\|\|\updatediterate-x\| \\ \nonumber
    &\quad -\|\widetilde{F}(\iterate,\xi_{t\tau_k})-  F(\iterate)\|\|\updatediterate-\iterate\| + \left\langle  \widetilde{F}(\pastiterate,\xi_{t\tau_k})- F(\pastiterate), \pastiterate-x \right\rangle \\ 
    &\quad-\sum_{t' = t-\kappa}^{t-1} \eta_k \bigg(2\bar{L}^2\|x_{t'}^{(k)}-x\|^2+ 2\|\widetilde{F}(x,\xi_{t'\tau_k})\|^2 +2\bar{L}^2\|\pastiterate-x\|^2 + 2\|\widetilde{F}(\iterate,\xi_{t\tau_k})-  F(\iterate)\|^2 \bigg).\label{eq:intermediate_base_bound}
\end{align}

Substituting the bound from Ineq.~\eqref{eq:intermediate_base_bound} into the three point lemma of Ineq.~\eqref{eq:base_eq}, we obtain
\begin{align}
\nonumber
      &\frac{1}{2}\|\iterate-x\|^2 - \frac{1}{2}\|\updatediterate-x\|^2 \\ \nonumber &\geq  \frac{1}{2} \|\updatediterate - \iterate\|^2 + \eta_k \left\langle F(\updatediterate), \updatediterate-x \right\rangle - \eta_k L \|\updatediterate-\iterate\|\|\updatediterate-x\|\\ \nonumber
    &\quad -\eta_k \|\widetilde{F}(\iterate,\xi_{t\tau_k})-  F(\iterate)\|\|\updatediterate-\iterate\| + \eta_k \left \langle  \widetilde{F}(\pastiterate,\xi_{t\tau_k})- F(\pastiterate), \pastiterate-x \right \rangle \\ 
    &\quad-\sum_{t' = t-\kappa}^{t-1} \eta_k^2 \bigg(2\bar{L}^2\|x_{t'}^{(k)}-x\|^2+ 2\|\widetilde{F}(x,\xi_{t'\tau_k})\|^2 +2\bar{L}^2\|\pastiterate-x\|^2 + 2\|\widetilde{F}(\iterate,\xi_{t\tau_k})-  F(\iterate)\|^2 \bigg).\label{eq:inter_recursion}
\end{align}
Using Young's inequality we have
 \begin{align}
 \nonumber
     &\frac{1}{2} \|\updatediterate - \iterate\|^2- \eta_k L \|\updatediterate-\iterate\|\|\updatediterate-x\| -\eta_k \|\widetilde{F}(\iterate,\xi_{t\tau_k})-  F(\iterate)\|\|\updatediterate-\iterate\| \\ \nonumber
     &\geq \frac{1}{2} \|\updatediterate - \iterate\|^2 - \frac{1}{4} \|\updatediterate - \iterate\|^2-\eta_k^2 L^2\|\updatediterate-x\|^2 - \frac{1}{4} \|\updatediterate - \iterate\|^2\\
     \nonumber&\qquad - \eta_k^2\|\widetilde{F}(\iterate,\xi_{t\tau_k})-  F(\iterate)\|^2 \\ \nonumber
     &\overset{\1}{\geq} -\eta_k^2\bar{L}^2\|\updatediterate-x\|^2 - \eta_k^2\|\widetilde{F}(\iterate,\xi_{t\tau_k})-  F(\iterate)\|^2,
 \end{align}
 where step $\1$ follows from the fact $-L \geq -\bar{L}$.
 Substituting this bound into Ineq.~\eqref{eq:inter_recursion} we obtain
 \begin{align}
     \nonumber
      \frac{1}{2}\|\iterate-x\|^2 t   &\geq  \left(\frac{1}{2} -\eta_k^2 \bar{L}^2 \right)\|\updatediterate - x\|^2 + \eta_k \left \langle F(\updatediterate), \updatediterate-x \right\rangle + \eta_k  \left\langle  \widetilde{F}(\pastiterate,\xi_{t\tau_k})- F(\pastiterate), \pastiterate-x \right \rangle \\ \nonumber
    &\qquad - \eta_k^2\|\widetilde{F}(\iterate,\xi_{t\tau_k})-  F(\iterate)\|^2-\sum_{t' = t-\kappa}^{t-1} \eta_k^2 \bigg(2\bar{L}^2\|x_{t'}^{(k)}-x\|^2+ 2\|\widetilde{F}(x,\xi_{t'\tau_k})\|^2 +2\bar{L}^2\|\pastiterate-x\|^2 \\ \label{eq:inter_recursion1}
    &\qquad \qquad \qquad \qquad \qquad \qquad \qquad \qquad \qquad \qquad \qquad+ 2\|\widetilde{F}(\iterate,\xi_{t\tau_k})-  F(\iterate)\|^2 \bigg).
 \end{align}
 Using monotonicity from Ineq.~\eqref{eq:monotonicity} we have $$\eta_k \left \langle F(\updatediterate), \updatediterate-x^{*} \right\rangle \geq \eta_k \mu \|\updatediterate - x^{*}\|^2.$$ Taking $x = x^*$ in Ineq.~\eqref{eq:inter_recursion1} and using the above bound we obtain
\begin{align}
    \nonumber
      &\frac{1}{2}\|\iterate-x^*\|^2  \\ \nonumber &\geq  \left(\frac{1}{2} -\eta_k^2 \bar{L}^2 + \eta_k \mu \right)\|\updatediterate - x^*\|^2  + \eta_k \left \langle  \widetilde{F}(\pastiterate,\xi_{t\tau_k})- F(\pastiterate), \pastiterate-x^* \right\rangle\\
      \nonumber &\quad- \eta_k^2\|\widetilde{F}(\iterate,\xi_{t\tau_k})-  F(\iterate)\|^2 \\ 
    &\quad -\sum_{t' = t-\kappa}^{t-1} \eta_k^2 \bigg(2\bar{L}^2\|x_{t'}^{(k)}-x^*\|^2+ 2\|\widetilde{F}(x^*,\xi_{t'\tau_k})\|^2 +2\bar{L}^2\|\pastiterate-x^*\|^2 +2 \|\widetilde{F}(\iterate,\xi_{t\tau_k})-  F(\iterate)\|^2 \bigg).\label{eq:inter_recursion2}
\end{align}
Taking expectation in Ineq.~\eqref{eq:inter_recursion2} and using Young's inequality we obtain
\begin{align}
    \nonumber
      &\frac{1}{2}\EE\|\iterate-x^*\|^2  \\ \nonumber 
      &\geq  \left(\frac{1}{2} -\eta_k^2 \bar{L}^2 + \eta_k \mu \right)\EE\|\updatediterate - x^*\|^2  + \eta_k  \EE \left \langle  \widetilde{F}(\pastiterate,\xi_{t\tau_k})- F(\pastiterate), \pastiterate-x^* \right\rangle\\
      \nonumber &\qquad - \eta_k^2\EE\|\widetilde{F}(\iterate,\xi_{t\tau_k})-  F(\iterate)\|^2 \\ \nonumber
    &\qquad -\sum_{t' = t-\kappa}^{t-1} \eta_k^2 \bigg(2\bar{L}^2\EE\|x_{t'}^{(k)}-x^*\|^2+ 4\EE\|\widetilde{F}(x^*,\xi_{t'\tau_k})-F(x^*)\|^2\\ \label{eq:inter_recursion3}
    &\qquad \qquad \qquad \qquad +4\|F(x^*)\|^2 +2\bar{L}^2\EE\|\pastiterate-x^*\|^2 
    + 2\EE\|\widetilde{F}(\iterate,\xi_{t\tau_k})-  F(\iterate)\|^2 \bigg).
\end{align}

Substituting Lemma~\ref{lemma:mixing_norm}, and Lemma \ref{bias_variance_lemma} with scaling of $\tau_k$ into Ineq.~\eqref{eq:inter_recursion3} we obtain
\begin{align}
    \nonumber
      &\frac{1}{2}\EE\|\iterate-x^*\|^2  \\ \nonumber &\geq  \left(\frac{1}{2} -\eta_k^2 \bar{L}^2 + \eta_k \mu \right)\EE\|\updatediterate - x^*\|^2  - \eta_k \cp \rho^{(\kappa+1) \tau_k-1}  \EE\|\pastiterate-x^*\|^2 - \eta_k \cm \rho^{(\kappa+1) \tau_k-1}\EE\|\pastiterate-x^*\| \\ \nonumber
      & \qquad - \eta_k^2\left(\sigma^2 + 4 \cm^2 \rho^{2(\tau_k-1)}  + \left(\zeta^2 + 8\bar{L}^2\right)\EE\|\iterate-x^*\|^2\right)\\ \nonumber
    &\qquad -\sum_{t' = t-\kappa}^{t-1} \eta_k^2 \bigg(2\bar{L}^2\EE\|x_{t'}^{(k)}-x^*\|^2+ 4\sigma^2+16 \cm^2 \rho^{2(\tau_k-1)}+4\|F(x^*)\|^2 +2\bar{L}^2\EE\|\pastiterate-x^*\|^2 \\ 
     & \qquad \qquad \qquad \qquad +2\sigma^2+8\cm^2 \rho^{2(\tau_k-1)}+ 2\left(\zeta^2+ 8\bar{L}^2 \right)\EE\|\iterate-x^*\|^2 \bigg).
\end{align}
By rearranging the terms, we obtain
\begin{align}
\nonumber
    &\left(\frac{1}{2} -\eta_k^2 \bar{L}^2 + \eta_k \mu \right)\EE\|\updatediterate - x^*\|^2 \\ \nonumber &\leq \left( \frac{1}{2} + \sum_{t' = t-\kappa}^{t} 2\eta_k^2(\zeta^2 + 8\bar{L}^2)  \right)\EE\|\iterate-x^*\|^2 + \eta_k \cp \rho^{(\kappa+1) \tau_k-1}  \EE\|\pastiterate-x^*\|^2\\ \nonumber &\qquad+  \eta_k \cm \rho^{(\kappa+1) \tau_k-1}  \EE\|\pastiterate-x^*\|\ +\sum_{t' = t-\kappa}^{t-1} \eta_k^2 (2\bar{L}^2\EE\|x_{t'}^{(k)}-x^*\|^2+2\bar{L}^2\EE\|\pastiterate-x^*\|^2) \\ \nonumber
    &\qquad+ \sum_{t' = t-\kappa}^{t} \eta_k^2 (6\sigma^2+24\cm^2 \rho^{2(\tau_k-1)} + 4 \|F(x^*)\|^2). \\ \nonumber
\end{align}
By applying Young's inequality and Jensen's inequality to the inequality above, we have
\begin{align}
\nonumber
    &\left(\frac{1}{2} -\eta_k^2 \bar{L}^2 + \eta_k \mu \right)\EE\|\updatediterate - x^*\|^2 \\ \nonumber
    &\overset{}{\leq}\left( \frac{1}{2} + \left(\kappa+1\right)2\eta_k^2\left(\zeta^2 + 8\bar{L}^2\right)  \right)\EE\|\iterate-x^*\|^2 +  \eta_k \cp \rho^{(\kappa+1) \tau_k-1}  \EE\|\pastiterate-x^*\|^2 \\ \nonumber &\qquad+ 10\eta_k \cm \rho^{(\kappa+1) \tau_k-1}+\frac{\eta_k\cm \rho^{(\kappa+1) \tau_k-1}\EE\|\pastiterate-x^*\|^2}{40}\\
    \nonumber &\qquad +\sum_{t' = t-\kappa}^{t-1} \eta_k^2 \left(2\bar{L}^2\EE\|x_{t'}^{(k)}-x^*\|^2+2\bar{L}^2\EE\|\pastiterate-x^*\|^2\right) \\ \label{eq:intermediate_check}
    &\qquad+ \left(\kappa+1\right) \eta_k^2 \left(6\sigma^2+24\cm^2 \rho^{2(\tau_k-1)}+ 4 \|F(x^*)\|^2\right).
\end{align}
 Now we set the value of the free parameter $\kappa$, which decides how much into the past we condition on. We set $\kappa = \tau_M/\tau_k $ if $t \geq \tau_M/\tau_k +1$, and set $\kappa = t-1$ if $t \leq \tau_M/\tau_k$. Multiplying by $\Gamma_t = (1+\eta_k \mu)^t$ (Eq.~\eqref{eq:theta_defn}), using the definition of $C = \left(\frac{\cm}{40} + \cp \right)$ (Eq.~\eqref{eq:C_defn}) and summing from $t=1$ to $T_k$  we have
\begin{align*}
    &\sum_{t=1}^{T_k}\left(\frac{1}{2} -\eta_k^2 \bar{L}^2 + \eta_k \mu \right)\Gamma_t\EE\|\updatediterate - x^*\|^2 \\ \nonumber &\leq \sum_{t=1}^{T_k}\left( \frac{1}{2} + \left(\kappa+1\right)2\eta_k^2\left(\zeta^2 + 8\bar{L}^2\right)  \right)\Gamma_t \EE\|\iterate-x^*\|^2 \\ 
    &\qquad+ 10 \sum_{t=1}^{T_k} \Gamma_t \eta_k \cm \rho^{(\kappa+1) \tau_k-1}+\sum_{t=1}^{T_k} \eta_k C\rho^{(\kappa+1) \tau_k-1}\Gamma_t  \EE\|\pastiterate-x^*\|^2 \\ \nonumber &\qquad+\sum_{t=1}^{T_k}\sum_{t' = t-\kappa}^{t-1} \eta_k^2 \Gamma_t (2\bar{L}^2\EE\|x_{t'}^{(k)}-x^*\|^2+2\bar{L}^2\EE\|\pastiterate-x^*\|^2) \\
    &\qquad+ \sum_{t=1}^{T_k}\Gamma_t \left(\kappa+1\right) \eta_k^2 \left(6\sigma^2+25\cm^2 \rho^{2(\tau_k-1)}+ 4 \|F(x^*)\|^2\right).
\end{align*}

We now define some parameters $A_t$ and $G_t$ as follows:
\begin{align}\label{eq:defn_params} 
A_{t} &= \left(\frac{1}{2} -\eta_k^2 \bar{L}^2 + \eta_k \mu \right)\Gamma_{t}, \\ \nonumber
    G_t & = \left \{
\begin{array}{ll}
\Gamma_1\left(\frac{1}{2} + 2\eta_k^2E \right)+\eta_kC \sum\limits_{t'=1}^{\alpha_k+1}\rho^{\tau_k t'-1}\Gamma_{t'}+2\eta_k^2 \bar{L}^2\sum\limits_{t'=1}^{\alpha_k+1} \Gamma_{t'} +2\eta_k^2 \bar{L}^2\sum\limits_{t'=1}^{\alpha_k} (t'-1)\Gamma_{t'},
& t=1\\
\Gamma_t\left(\frac{1}{2} + 2\eta_k^2Et \right)  +\eta_k C\rho^{\tau_M+\tau_k-1}\Gamma_{t+\alpha_k}+ 2\eta_k^2 \bar{L}^2\sum\limits_{t' = t+1}^{t+\alpha_k}\Gamma_{t'}+2\eta_k^2\bar{L}^2\alpha_k\Gamma_{t+\alpha_k},                    & t \in [2,\alpha_k]\\
\Gamma_t\left(\frac{1}{2} + 2\eta_k^2E(\alpha_k+1) \right)  +\eta_k C\rho^{\tau_M+\tau_k-1}\Gamma_{t+\alpha_k}+ 2\eta_k^2 \bar{L}^2\sum\limits_{t' = t+1}^{t+\alpha_k}\Gamma_{t'}+2\eta_k^2\bar{L}^2\alpha_k\Gamma_{t+\alpha_k},&t \in (\alpha_k, T_k-\alpha_k]\\
\Gamma_t\left(\frac{1}{2} + 2\eta_k^2E(\alpha_k+1) \right) +2\eta_k^2 \bar{L}^2\sum\limits_{t' = t+1}^{T_k}\Gamma_{t'} ,&t \in (T_k-\alpha_k, T_k]
\end{array}
\right.\nonumber
\end{align}
Recall from Proposition~\ref{recursion_proposition} that $H = \sum_{t=1}^{\alpha_k}t\Gamma_t+\sum_{t=\alpha_k+1}^{T_k}(\alpha_k+1)\Gamma_t$.

Substituting the definitions of $A_t$ and $G_t$ from Eq.~\eqref{eq:defn_params} and $H$ we obtain
\begin{align*}
    \sum_{t=1}^{T_k}A_t\EE\|\updatediterate - x^*\|^2 &\leq \sum_{t=1}^{T_k}G_t \EE\|\iterate-x^*\|^2 + 10\sum_{t=1}^{T_k} \Gamma_t \eta_k \cm \rho^{(\kappa+1) \tau_k-1}\\
    &\qquad + H \eta_k^2\left(6\sigma^2+25\cm^2 \rho^{2(\tau_k-1)}+ 4 \|F(x^*)\|^2\right).
\end{align*}
The definition for $G_t$ is obtained by simple algebraic manipulations and counting the number of coefficients in the different regimes of the time step $t$. The definitions of $A_t$ and $H$ are clear from the above equation.
We claim that $G_{t+1} \leq A_{t}$, which we verify in Section~\ref{subsubsec:verify_condition}. Using this condition we obtain 
\begin{align*}
    A_{T_k}&\EE\|\finaliterate - x^*\|^2\\
    &\leq G_1 \|\initialiterate - x^*\|^2+ 10\sum_{t=1}^{T_k} \Gamma_t \eta_k \cm \rho^{(\kappa+1) \tau_k-1}+ H \eta_k^2\left(6\sigma^2+25\cm^2 \rho^{2(\tau_k-1)}+ 4 \|F(x^*)\|^2\right).
\end{align*}
Expanding the summation term in the above bound which appears due to the bias at the optimal solution $x^*$, we obtain
\begin{align*}
    10\sum_{t=1}^{T_k} \Gamma_t \eta_k \cm \rho^{(\kappa+1) \tau_k-1}= 10 \eta_k \cm\sum_{t=1}^{\tau_M/\tau_k}\rho^{t\tau_k-1}\Gamma_t + 10 \eta_k \cm\sum_{t=\tau_M/\tau_k + 1}^{T_k}\rho^{\tau_M+\tau_k-1}\Gamma_t.
\end{align*}
Finally, we arrive at
\begin{align}
\nonumber
    \EE\|\finaliterate - x^*\|^2&\leq \frac{G_1}{A_{T_k}} \|\initialiterate - x^*\|^2+ \frac{10 \eta_k \cm}{A_{T_k}}\left(\sum_{t=1}^{\tau_M/\tau_k}\rho^{t\tau_k-1}\Gamma_t + \sum_{t=\tau_M/\tau_k + 1}^{T_k}\rho^{\tau_M+\tau_k-1}\Gamma_t\right) \\ \nonumber 
    &\qquad+ \frac{H}{A_{T_k}} \eta_k^2\left(6\sigma^2+25\cm^2 \rho^{2(\tau_k-1)}+ 4 \|F(x^*)\|^2\right),
\end{align}
as in Proposition~\ref{recursion_proposition}.

It remains to verify the condition $A_t\geq G_{t+1}$ and to prove Lemma~\ref{lemma:three_point_MER} and Claim~\ref{claim:inner_prod}.
\subsubsection{\texorpdfstring{Verifying $A_t\geq G_{t+1}$}{}}\label{subsubsec:verify_condition}
Now we verify the condition $A_t \geq G_{t+1}$ is satisfied for the given choice of  parameters $\eta_k$ and $\Gamma_t$. From the definitions in Proposition~\ref{recursion_proposition} we have
\begin{align*}
    A_t &= \left(\frac{1}{2} -\eta_k^2 \bar{L}^2 + \eta_k \mu \right)\Gamma_t \text{ and for all } t\in[1,T_k-1] \\ 
    G_{t+1} &\leq \Gamma_{t+1}\left(\frac{1}{2} + 2\eta_k^2E\left(\alpha_k+1\right) \right)  +\eta_k C \rho^{\tau_M+\tau_k-1}  \Gamma_{t+1+\alpha_k}+ 2\eta_k^2 \bar{L}^2\sum\limits_{t' = t+2}^{t+1+\alpha_k}\Gamma_{t'}+2\eta_k^2\bar{L}^2\alpha_k\Gamma_{t+1+\alpha_k} \\ \nonumber
    &\overset{\1}{\leq} \Gamma_{t+1}\left(\frac{1}{2} + \eta_k^2E\left(\alpha_k+1\right) \right)  +\eta_k C \rho^{\tau_M+\tau_k-1} \Gamma_{t+1+\alpha_k}+ 2\eta_k^2 \bar{L}^2\alpha_k\Gamma_{t+1+\alpha_k}+2\eta_k^2\bar{L}^2\alpha_k\Gamma_{t+1+\alpha_k} \\ \nonumber\
    &\overset{\2}{\leq} \Gamma_{t+1}\left(\frac{1}{2} + \eta_k^2(\zeta^2+8 \bar{L}^2) \left(\alpha_k+1\right) \right)+\eta_k C \rho^{\tau_M+\tau_k-1} \Gamma_{t+1+\alpha_k}+ 4\eta_k^2 \bar{L}^2\alpha_k\Gamma_{t+1+\alpha_k},
\end{align*}
where step $\1$ follows from the increasing property of the parameter $\Gamma_t$ defined in Eq.~\eqref{eq:definitions} and step $\2$ follows by substituting the definition of $E=\zeta^2 + 8 \bar{L}^2$. In order to show $A_t \geq G_{t+1}$ we proceed as follows:
\begin{align}
\nonumber
    A_t -G_{t+1} &\geq \left(\frac{1}{2} -\eta_k^2 \bar{L}^2 + \eta_k \mu \right)\Gamma_t - \Gamma_{t+1}\left(\frac{1}{2} + \eta_k^2(\zeta^2+8 \bar{L}^2) \left(\alpha_k+1\right) \right) \\ \nonumber
    &\qquad-\eta_k C \rho^{\tau_M+\tau_k-1} \Gamma_{t+1+\alpha_k}- 4\eta_k^2 \bar{L}^2\alpha_k\Gamma_{t+1+\alpha_k} \\ \nonumber
    &\overset{\1}{\geq} \Gamma_{t}\left( \frac{\eta_k \mu}{2} -\eta_k^2 \bar{L}^2 - \eta_k^2(\zeta^2+8 \bar{L}^2) \left(\alpha_k+1\right)\Gamma-\eta_k C \rho^{\tau_M+\tau_k-1} \Gamma_{\alpha_k+1}- 4\eta_k^2 \bar{L}^2\alpha_k\Gamma_{\alpha_k+1}  \right) \\ \nonumber
    &\overset{}{\geq} \Gamma_t\left( \frac{\eta_k \mu}{2}  - \eta_k^2(\zeta^2+13 \bar{L}^2) \left(\alpha_k+1 \right)\Gamma_{\alpha_k+1}-\eta_k C \rho^{\tau_M+\tau_k-1} \Gamma_{\alpha_k+1}  \right) \\ \label{eq:intermediate_eq}
    &\overset{\2}{\geq} \Gamma_t\left( \frac{\eta_k \mu}{2}   -\frac{3 \eta_k \mu}{8}\Gamma_{\alpha_k+1} -\frac{ \eta_k \mu}{18} \Gamma_{\alpha_k+1}  \right),
\end{align}
where step $\1$ follows from the definition of $\Gamma_t$ and step $\2$ follows from the condition on $\eta_k$ discussed in Ineq.~\eqref{eq:eta_observation} and the condition $\tau_M = \frac{\log 18C/\mu}{\log 1/\rho}$.
Substituting the bound for $\Gamma_{\alpha_k+1}$ from Ineq.~\eqref{eq:theta_bound} into Ineq.~\eqref{eq:intermediate_eq}, we obtain
\begin{align*}
    A_t -G_{t+1} \geq 0.
\end{align*}
Thus we have verified the condition $A_t \geq G_{t+1}$ is satisfied for given choice of parameters.

\subsubsection{Proof of Lemma~\ref{lemma:three_point_MER}}\label{sec:pf_lemma_three_point_MER}
Let us define a function $\Phi(x) \defn  \eta_k \langle \widetilde{F}(\iterate,\xi_{t\tau_k}), x \rangle + \frac{1}{2} \|\iterate - x\|^2$. Using the strong convexity of $\Phi(x)$ we have
\begin{align}
    \Phi(x) \geq \Phi(\updatediterate)+ \langle \nabla \Phi(\updatediterate),x-\updatediterate \rangle + \frac{1}{2}\|\updatediterate-x\|^2.
\end{align}
Using the optimality condition for $\Phi(x)$ we have $\langle \nabla \Phi(\updatediterate),x-\updatediterate \rangle \geq 0$. Thus
\begin{align}
    \Phi(x) \geq \Phi(\updatediterate)+ \frac{1}{2}\|\updatediterate-x\|^2.
\end{align}
Substituting the definition for $\Phi(x)$ we get the result stated in Lemma \ref{lemma:three_point_MER}.

\subsubsection{Proof of Claim~\ref{claim:inner_prod}}\label{sec:pf_claim_inner_prod}
We now focus on bounding the last term in Ineq.~\eqref{eq:intermediate_inner_product_bound}. 

We begin by decomposing $\left \langle  \widetilde{F}(\iterate,\xi_{t\tau_k})-  F(\iterate), \updatediterate-x \right \rangle$ using a past iterate $\pastiterate$ for some parameter $\kappa \in \{0,1,2,\ldots\}$. This is an analysis tool introduced in order to handle the Markovian correlations in the data. We thus have
\begin{align}
\nonumber
     &\left \langle  \widetilde{F}(\iterate,\xi_{t\tau_k})-  F(\iterate), \updatediterate-x \right \rangle \\ \nonumber
     &\quad=  \left \langle  \widetilde{F}(\iterate,\xi_{t\tau_k})-  F(\iterate), \updatediterate-\pastiterate \right\rangle +  \left\langle  \widetilde{F}(\iterate,\xi_{t\tau_k})- \widetilde{F}(\pastiterate,\xi_{t \tau_k}), \pastiterate-x \right\rangle \\ \nonumber
     & \qquad + \left\langle  \widetilde{F}(\pastiterate,\xi_{t\tau_k})- F(\pastiterate), \pastiterate-x \right\rangle  +  \left \langle  F(\pastiterate)- F(\iterate), \pastiterate-x \right \rangle \\ \nonumber 
     &\quad \overset{\1}{\geq} - \|\widetilde{F}(\iterate,\xi_{t\tau_k})-  F(\iterate)\|\|\updatediterate-\pastiterate\|  - (\widetilde{L}_1+L)\|\pastiterate-\iterate\|\|\pastiterate-x\| \\ \nonumber
     &\qquad + \big\langle  \widetilde{F}(\pastiterate,\xi_{t\tau_k})- F(\pastiterate), \pastiterate-x \big\rangle \\ \nonumber
     &\quad \overset{\2}{=} - \|\widetilde{F}(\iterate,\xi_{t\tau_k})-  F(\iterate)\|\|\updatediterate-\pastiterate\|  - \bar{L}\|\pastiterate-\iterate\|\|\pastiterate-x\| \\ \label{eq:intermediate_bd}
     &\qquad + \left\langle  \widetilde{F}(\pastiterate,\xi_{t\tau_k})- F(\pastiterate), \pastiterate-x \right \rangle,
\end{align}
where step $\1$ follows from Cauchy--Schwarz inequality and the Lipschitz assumption in Ineq.~\eqref{eq:F_lipschitz} and Ineq.~\eqref{eq:lipscitz1}. In step $\2$ we use the parameter $\bar{L} = \widetilde{L}_1+ L$, defined in Eq.~\eqref{eq:definitions}.

We now bound the terms in Ineq.~\eqref{eq:intermediate_bd} separately.
We claim the following inequality and prove it at the end of the section:
\begin{align} \label{contraction_lemma}
        \|\updatediterate-\iterate\| \leq 2 \eta_k \|\widetilde{F}(\iterate,\xi_{t \tau_{k}})\|.
\end{align}
Using Eq.~\eqref{contraction_lemma} and the triangle inequality, we bound $\|\iterate - \pastiterate\|$ as follows:
\begin{align}
\label{eq:contraction+triangle}
    \|\iterate - \pastiterate\| \leq \sum_{t' = t-\kappa}^{t-1} 2\eta_k \|\widetilde{F}(x_{t'}^{(k)},\xi_{t' \tau_{k}})\| \overset{\1}{\leq} \sum_{t' = t-\kappa}^{t-1} 2\eta_k \left(\bar{L}\|x_{t'}^{(k)}-x\| + \|\widetilde{F}(x,\xi_{t' \tau_{k}})\|\right),
\end{align}
where step $\1$ follows from the Lipschitz assumption and the fact that $\widetilde{L}_1 \leq \bar{L}$.
Consequently,
\begin{align}
\nonumber
    \|\iterate - \pastiterate\|\|\pastiterate - x\| &\leq \sum_{t'=t-\kappa}^{t-1}2\eta_k\left(\bar{L}\|x_{t'}^{(k)}-x\|\|\pastiterate - x\| + \|\widetilde{F}(x,\xi_{t' \tau_{k}})\|\|\pastiterate - x\|\right) \\ \label{eq:product_bound}
    &\overset{\1}{\leq} \sum_{t' = t-\kappa}^{t-1} \eta_k \left( \frac{1}{\bar{L}}\|\widetilde{F}(x,\xi_{t'\tau_k})\|^2 +\bar{L} \|x_{t'}^{(k)}-x\|^2+2\bar{L} \|\pastiterate-x\|^2\right)
\end{align}
where step $\1$ follows from the Young's inequality.

We also have
\begin{align}
\nonumber
    &\|\widetilde{F}(\iterate,\xi_{t\tau_k})-  F(\iterate)\|\|\updatediterate-\pastiterate\| \\ \nonumber 
    & \leq \|\widetilde{F}(\iterate,\xi_{t\tau_k})-  F(\iterate)\| \left(\|\updatediterate-\iterate\|+\|\iterate-\pastiterate\| \right) \\ \nonumber
    &\overset{\1}{\leq} \|\widetilde{F}(\iterate,\xi_{t\tau_k})-  F(\iterate)\|\|\updatediterate-\iterate\| \\ \nonumber
    & \qquad+ \sum_{t' = t-\kappa}^{t-1} 2\eta_k\|\widetilde{F}(\iterate,\xi_{t\tau_k})-  F(\iterate)\| \left(\bar{L}\|x_{t'}^{(k)}-x\| + \|\widetilde{F}(x,\xi_{t' \tau_{k}})\|\right) \\ \nonumber
    &\overset{\2}{\leq} \|\widetilde{F}(\iterate,\xi_{t\tau_k})-  F(\iterate)\|\|\updatediterate-\iterate\| \\ \label{eq:intermediate_product_bound}
    & \qquad + \sum_{t' = t-\kappa}^{t-1} \eta_k \left(\bar{L}^2\|x_{t'}^{(k)}-x\|^2 + 2\|\widetilde{F}(\iterate,\xi_{t\tau_k})-  F(\iterate)\|^2 +  \|\widetilde{F}(x,\xi_{t' \tau_{k}})\|^2  \right).
\end{align}
where step $\1$ follows from the triangle inequality and Ineq.~\eqref{eq:contraction+triangle} while step $\2$ follows from Young's inequality.

Substituting the bounds from Ineq.~\eqref{eq:product_bound} and Ineq.~\eqref{eq:intermediate_product_bound} into Ineq.~\eqref{eq:intermediate_bd} we obtain
\begin{align*}
     U_3 &\geq -\|\widetilde{F}(\iterate,\xi_{t\tau_k})-  F(\iterate)\|\|\updatediterate-\iterate\| \\ 
    & \qquad - \sum_{t' = t-\kappa}^{t-1} \eta_k \left(\bar{L}^2\|x_{t'}^{(k)}-x\|^2 + 2\|\widetilde{F}(\iterate,\xi_{t\tau_k})-  F(\iterate)\|^2 +  \|\widetilde{F}(x,\xi_{t' \tau_{k}})\|^2  \right) \\ \nonumber
    & \qquad -\sum_{t' = t-\kappa}^{t-1} \eta_k \bar{L}\left( \frac{1}{\bar{L}}\|\widetilde{F}(x,\xi_{t'\tau_k})\|^2 +\bar{L} \|x_{t'}^{(k)}-x\|^2+2\bar{L} \|\pastiterate-x\|^2\right) \\ 
    & \qquad +\left \langle  \widetilde{F}(\pastiterate,\xi_{t\tau_k})- F(\pastiterate), \pastiterate-x \right\rangle \\ 
    &= -\|\widetilde{F}(\iterate,\xi_{t\tau_k})-  F(\iterate)\|\|\updatediterate-\iterate\| + \left\langle  \widetilde{F}(\pastiterate,\xi_{t\tau_k})- F(\pastiterate), \pastiterate-x \right\rangle \\
    &\qquad-\sum_{t' = t-\kappa}^{t-1} \eta_k \bigg(2\bar{L}^2\|x_{t'}^{(k)}-x\|^2+ 2\|\widetilde{F}(x,\xi_{t'\tau_k})\|^2 +2\bar{L}^2\|\pastiterate-x\|^2 \\ \nonumber
    &\qquad \qquad \qquad \qquad+ 2\|\widetilde{F}(\iterate,\xi_{t\tau_k})-  F(\iterate)\|^2 \bigg),
\end{align*} 
which completes the proof of Claim~\ref{claim:inner_prod}. It only remains to establish claim~\eqref{contraction_lemma}.

\paragraph{Proof of claim~\eqref{contraction_lemma}.}
 From Eq.~\eqref{eq:MER_update} we have
\begin{align*}
    \left \langle \eta_k \widetilde{F}(\iterate,\xi_{t\tau_k}),\updatediterate \right \rangle+\frac{1}{2} \|\updatediterate-\iterate\|^2  \leq \left \langle \eta_k \widetilde{F}(\iterate,\xi_{t\tau_k}),x \right \rangle+\frac{1}{2} \|x-\iterate\|^2 \text{ for all } x \in X.
\end{align*}
Substituting $x = \iterate$ we obtain
\begin{align*}
    \left \langle \eta_k\widetilde{F}(\iterate,\xi_{t\tau_k}),\updatediterate-\iterate \right \rangle+\frac{1}{2} \|\updatediterate-\iterate\|^2  \leq 0.
\end{align*}
Applying the Cauchy--Schwarz inequality yields
\begin{align*}
    -\eta_k \| \widetilde{F}(\iterate,\xi_{t\tau_k})\| \|\updatediterate-\iterate\| \leq -\frac{1}{2} \|\updatediterate-\iterate\|^2. 
\end{align*}
Rearranging the terms gives us the contraction lemma.

This concludes the proof of Proposition~\ref{recursion_proposition} including its auxiliary lemmas and claims.

\subsection{Proof of Theorem \ref{thm:emulate_iid}}\label{sec:pf_thm_emulate_iid}
In order to prove Theorem \ref{thm:emulate_iid}, it is useful to formally define the serial skipped experience replay (SSER) algorithm. This algorithm simply tracks the epoch of the MER algorithm in  which $\tau_k = \beta \tau_M$. Recalling the notation from Section~\ref{sec:following_iid}, let $x_{T+1}$ denote the iterate obtained after applying $T$-steps of the SSER update of Eq.~\eqref{eq:constrained_update} starting from some initialization. Conditioned on the initialization, the randomness in $x_T$ is due to the samples $ \{\xi_{t \beta \tau_M}\}_{t=1}^{T}$ used to update the iterates. Similarly, we define $\widetilde{x}_{T+1}$ as the iterate obtained by applying $T$-steps of SA on the i.i.d. sequence of samples $\{\widetilde{\xi}_{t \beta \tau_M}\}_{t=1}^{T}$, starting from the same initialization. Recall the error terms, $\Delta_{T+1} = x_{T+1}-x^*$ and $\widetilde{\Delta}_{T+1} = \widetilde{x}_{T+1}-x^*$. Both algorithms are initialized at $x_1$, so that $\widetilde{x}_1 = x_1$. We let $T = \frac{B}{\beta \tau_M}$.
\begin{algorithm}[H]
\caption{Serial Skipped Experience Replay}\label{alg:three}
\begin{algorithmic}
\STATE{\textbf{Input}: Memory Buffer $\{\xi_1,\xi_1,\ldots,\xi_{B}\}$, Effective mixing time of Markov chain $\tau_M$}, initialization $x_1$.
    \FOR{$t=1, \ldots, T$}
    \STATE{Using the sample $\xi_{t \beta \tau_M}$, perform one step of SA:}
    \STATE{
    \begin{align}\label{eq:constrained_update}
    x_{t+1} = \arg \min_{x \in X} \eta \langle \widetilde{F}(x_{t},\xi_{t\beta \tau_M}),x \rangle + \frac{1}{2}\|x_{t}-x\|^2.
    \end{align}}
\ENDFOR
\end{algorithmic}
\end{algorithm}
Having set up the notation, we proceed with our proof.
We define auxiliary sequences with unconstrained updates, as
\begin{align}\label{eq:unconstrained_update}
    x_{t+1}' &= x_{t} - \eta \widetilde{F}(x_{t},\xi_{t\beta \tau_M}),\text{ for } t=1,\ldots,T\nonumber\\
    \widetilde x_{t+1}' &= \widetilde x_{t} - \eta \widetilde{F}(\widetilde x_{t},\xi_{t\beta \tau_M}),\text{ for } t=1,\ldots,T.
\end{align}
Thus, at time step $t$, the iterate corresponding to $x_{t}$ before taking the projection is $x_{t}'$. Since $X$ is a closed convex set and the projection onto such sets are contractive, we have
\begin{align}\label{eq:convex_contraction}
    \|x_{t+1} - \widetilde{x}_{t+1}\| \leq \|x_{t+1}' - \widetilde{x}_{t+1}'\|,\text{ for all }t = 1,\ldots,T.
\end{align}
Next, we show that the function $h$ mapping the random variables $(\xi_{j \beta \tau_M})_{j = 1}^{T }$ to $\|\Delta_{T+1}\|$, which maps the states of a Markov chain to the error, is a Lipschitz continuous function. We have
\begin{align*}
    \|x_{T+1}-\widetilde{x}_{T+1}\| &\overset{\1}{\leq} \|x_{T+1}' - \widetilde{x}_{T+1}'\|\\ \nonumber &\overset{\2}{=} \|x_{T}-\eta \widetilde{F}(x_{T},\xi_{T\beta \tau_M})-\widetilde{x}_{T}+\eta \widetilde{F}(\widetilde{x}_{T},\widetilde{\xi}_{T\beta \tau_M})\| \\ 
    &\overset{}{\leq} \|x_{T}-\widetilde{x}_{T}\| + \eta \| \widetilde{F}(x_{T},\xi_{T\beta \tau_M})-\widetilde{F}(\widetilde{x}_{T},\widetilde{\xi}_{T\beta \tau_M})\| \\ 
    &\overset{\3}{\leq} (1+\eta \widetilde{L}_1) \|x_{T} -\widetilde{x}_{T}\| + \eta \widetilde{L}_2 \|\xi_{T\beta \tau_M}-\widetilde{\xi}_{T\beta \tau_M}\|,
\end{align*}
where step $\1$ follows from Ineq.~\eqref{eq:convex_contraction}, step $\2$ follows from the definition in Eq.~\eqref{eq:unconstrained_update} and step $\3$ follows from Assumption~\ref{assmp:lipscitz1} and Assumption~\ref{assmp:lipscitz2}.
Applying the above relation recursively yields
\begin{align*}
     \|x_{T+1} - \widetilde{x}_{T+1}\|&\overset{\1}{\leq} \eta \widetilde{L}_2  \sum_{i=1}^{T} \left((1+\eta \widetilde{L}_1)^{T-i}\|\xi_{i\beta \tau_M}-\widetilde{\xi}_{i\beta \tau_M}\|\right)   \\
     &\leq   \eta \widetilde{L}_2(1+\eta \widetilde{L}_1)^{T-1}\sum_{i=1}^{T} \left( \|\xi_{i\beta \tau_M}-\widetilde{\xi}_{i\beta \tau_M}\|\right),
\end{align*}
where step $\1$ follows from the fact that $x_1 = \widetilde{x}_1$. We thus have
\begin{align*}
    \left|\|\Delta_{T+1}\| - \|\widetilde{\Delta}_{T+1}\|\right| &\overset{\1}{\leq} \|\Delta_{T+1}-\widetilde{\Delta}_{T+1}\|= \|x_{T+1} -\widetilde{x}_{T+1}\| \leq \eta \widetilde{L}_2(1+\eta \widetilde{L}_1)^{T-1}\sum_{i=1}^{T} \left( \|\xi_{i\beta \tau_M}-\widetilde{\xi}_{i\beta \tau_M}\|\right),
\end{align*}
where step $\1$ follows from the triangle inequality. Squaring both sides we obtain
\begin{align*}
    \left|\|\Delta_{T+1}\| - \|\widetilde{\Delta}_{T+1}\|\right|^2 \overset{\1}{\leq}T\eta^2 \widetilde{L}_2^2\left(1+\eta \widetilde{L}_1 \right)^{2(T-1)} \sum_{i=1}^{T}  \|\xi_{i\beta \tau_M}-\widetilde{\xi}_{i\beta \tau_M}\|^2,
\end{align*}
where step $\1$ follows from the Young's inequality.
This proves that the function $h$ is Lipschitz with constant $\sqrt{T}\eta \widetilde{L}_2\left(1+\eta \widetilde{L}_1 \right)^{T-1}$.

Now note that by Kantorovich duality, we have
\begin{align}
    \mathcal{W}_1(\mu,\nu) = \sup _{g \in \text{Lip}_1} \big| \EE_{Y \sim \mu}[g(Y)] - \EE_{Y \sim \nu}[g(Y)] \big|,
\end{align}
where Lip$_{1}$ denotes the set of all 1-Lipschitz continuous functions $g: \mathcal{X}\mapsto \mathbb{R}$. Since $h$ is $\sqrt{T}\eta \widetilde{L}_2\left(1+\eta \widetilde{L}_1 \right)^{T-1}$-Lipschitz continuous, applying the above inequality yields
\begin{align}\label{eq:lipschitz_bound_res}
    \left|\EE\|\Delta_{T+1}\| - \EE\|\widetilde{\Delta}_{T+1}\|\right| \leq \sqrt{T}\eta \widetilde{L}_2\left(1+\eta \widetilde{L}_1 \right)^{T-1} \cdot \mathcal{W}_1\left(\{\xi_{i\beta \tau_M}\}_{i=1}^{T},\{\widetilde{\xi}_{i\beta \tau_M}\}_{i=1}^{T}\right)
\end{align}
where $\mathcal{W}_1\left(\{\xi_{i\beta \tau_M}\}_{i=1}^{T},\{\widetilde{\xi}_{i\beta \tau_M}\}_{i=1}^{T}\right)$ denotes the Wasserstein-1 distance between the Markov chain used by SSER to update the iterates and the i.i.d. sample sequence. 
In the following lemma we bound this Wasserstein-1 distance.
\begin{lemma}\label{wasserstein_bound}
    Let $\tau_M$ be the effective mixing time of the Markov chain.
    Define the stochastic processes
    \begin{align*}
        Z = \{\xi_1,\xi_{\beta \tau_M},\xi_{2\beta \tau_M},...,\xi_{T\beta \tau_M}\}, \\ \nonumber
        \widetilde{Z} = \{\widetilde{\xi}_1,\widetilde{\xi}_{\beta \tau_M},\widetilde{\xi}_{2\beta \tau_M},...,\widetilde{\xi}_{T\beta \tau_M}\},
    \end{align*}
    where $\widetilde{\xi}_i \sim \pi$ is drawn from the stationary distribution independently from everything else. Then for come universal positive constant $c_0$ the following holds:
    \begin{align*}
        \mathcal{W}_{1}(Z,\widetilde{Z}) \leq c_0 Te^{-\beta}.
    \end{align*}
\end{lemma}
We defer the proof of this lemma to the end of this section, noting that similar lemmas have been established for TV mixing~\citep[see, e.g.][]{wingit, pmlr-v291-nakul25a}.
Substituting Lemma \ref{wasserstein_bound} into Ineq.~\eqref{eq:lipschitz_bound_res} we obtain
\begin{align*}
    \left|\EE\|\Delta_{T+1}\| - \EE\|\widetilde{\Delta}_{T+1}\|\right| \leq c_0 \eta T^{3/2}  \widetilde{L}_2\left(1+\eta \widetilde{L}_1 \right)^{T-1}e^{-\beta}.
\end{align*}
Finally, using the definition of the step-size $\eta \leq \frac{1}{T \widetilde{L}_1}$ completes the proof of Theorem \ref{thm:emulate_iid}.

\subsubsection{Proof of Lemma \ref{wasserstein_bound}}
In order to prove the lemma, let us first define an intermediate stochastic process:
\begin{align}
\nonumber
\widetilde{Z_i} = \{ \xi_1,...,\xi_{(i-1)\beta \tau_M},\widetilde{\xi}_{i\beta \tau_M},\xi_{(i+1)\beta \tau_M},...,\xi_{T\beta \tau_M} \}.
\end{align}
This stochastic process replaces the $i$-th sample from $Z$ with a sample from the stationary distribution, which is independent of everything else. Under Assumption \ref{assmp:wasserstein_mixing},
Lemma 4 in \citet{mou2024optimal} states that $\mathcal{W}_1(\xi_{i\beta \tau_M},\widetilde{\xi}_{i\beta \tau_M}) \leq c_0 e^{-\beta} \text{ for all } i=1,...,T$. From the triangle inequality for the Wasserstein metric, we have 
\begin{align}
    \mathcal{W}_{1}(Z,\widetilde{Z})\leq \sum_{i=1}^{T} \mathcal{W}_{1}(Z,\widetilde{Z}_{i}) \leq c_0 Te^{-\beta}.
\end{align}

\subsection{Proof of Corollary~\ref{corollary:GLM_1}}\label{sec:pf_corollary_GLM_1}
In order to prove Corollary~\ref{corollary:GLM_1} we proceed by verifying the necessary assumptions and then invoking the statement of Theorem~\ref{thm:main_error}.

First, we show that the stochastic operator defined in Eq.~\eqref{eq:glm_definition} is $L_f \cdot D_a^2$-Lipschitz continuous in its first argument:
\begin{align*}
    \|\widetilde{F}(x_1,\xi)-\widetilde{F}(x_2,\xi)\| &= \|a f(a^T x_1)-a y-a f(a^T x_2)+a y\| \\ \nonumber
    &=\|a f(a^T x_1)-a f(a^T x_2)\| \\ 
    &\overset{\1}{\leq} L_f \cdot \|a\|^2 \cdot \|x_1 - x_2\| \\
    &\overset{\2}{\leq} L_f \cdot  D_a^2 \cdot \|x_1 - x_2\|,
\end{align*}
where step $\1$ follows from the Lipschitz continuity of the link function $f$ and step $\2$ follows from the boundedness of the regressor trajectory, i.e. $\|a_t\| \leq D_a$ a.s. for all $t$. Thus, for GLMs, the parameter $\bar{L} = L + \widetilde{L}_1 = 2 \cdot L_f \cdot D_a^2$, and this verifies Assumption~\ref{assmp:lipscitz1}. 

We now verify Assumption~\ref{assmp:state_dependent_noise} for the GLM as follows:
\begin{align*}
    &\quad \EE\left[ \big\| \widetilde{F}(x,\xi_{t'})-\EE [ \widetilde{F}(x,\xi_{t'}) | \Fspace_{t-1}] \big\|^2 \big| \Fspace_{t-1} \right] \\ 
    &\overset{\1}{=}\EE\left[ \big\| a_{t'}(f(a_{t'}^Tx)-y_{t'})-\EE [ \widetilde{F}(x,\xi_{t'}) | \Fspace_{t-1}]+\EE [ \widetilde{F}(x^*,\xi_{t'}) | \Fspace_{t-1}]\big\|^2\big| \Fspace_{t-1} \right] \\
    &\overset{}{=}\EE\left[ \big\| a_{t'}(f(a_{t'}^Tx)-f(a_{t'}^Tx^*))+a_{t'}(f(a_{t'}^Tx^*)-y_{t'})-\EE [ a_{t'}(f(a_{t'}^Tx)-f(a_{t'}^Tx^*)) | \Fspace_{t-1}]\big\|^2\big| \Fspace_{t-1} \right] \\
    &\overset{\2}{\leq} 3\EE\left[ \big\| a_{t'}(f(a_{t'}^Tx)-f(a_{t'}^Tx^*))\big\|^2\big| \Fspace_{t-1} \right]+3\EE\left[\big\|a_{t'}(f(a_{t'}^Tx^*)-y_{t'})\big\|^2\big| \Fspace_{t-1} \right] \\ &\qquad+3\EE\left [ \big\|\EE[a_{t'}(f(a_{t'}^Tx)-f(a_{t'}^Tx^*)) | \Fspace_{t-1}]\big\|^2\big| \Fspace_{t-1} \right],
\end{align*}
where step $\1$ follows from the fact that for GLMs $\EE[\widetilde{F}(x^*,\xi_{t'})|\Fspace_{t-1}] = \EE[a_{t'}v_{t'}|\Fspace_{t-1}] =0$ since $\zeta_{t'}$ is zero mean i.i.d. noise and step $\2$ follows from Young's inequality. Now using the boundedness of the regressor trajectory and Lipschitzness of $f$ we have:
\begin{align*}
    &\quad \EE\left[ \big\| \widetilde{F}(x,\xi_{t'})-\EE [ \widetilde{F}(x,\xi_{t'}) | \Fspace_{t-1}] \big\|^2 \big| \Fspace_{t-1} \right] \\
    &\overset{}{\leq} 6\cdot L_f^2 \cdot D_a^4 \cdot \|x-x^{*}\|^2+3\EE\left[\big\|a_{t'}v_{t'}\big\|^2\big| \Fspace_{t-1} \right] \\
    &\leq 6\cdot L_f^2 \cdot D_a^4 \cdot \|x-x^{*}\|^2 + 3\cdot D_a^2\cdot \sigma_v^2.
\end{align*}
Thus Assumption~\ref{assmp:state_dependent_noise} is satisfied with $\sigma^2 = 6 D_a^2 \sigma_v^2$ and $\zeta^2 = 12 L_f^2 D_a^4$.

 For GLMs, $F(x^*) = 0$ and from Eq.~\eqref{eq:GLM_model} recall that we have $\EE[\widetilde{F}(x^*,\xi_{t'})|\Fspace_{t-1}] = \EE[a_{t'}v_{t'}|\Fspace_{t-1}] =0$. Thus Assumption~\ref{assmp:optimal_condition} is satisfied with $\cm=0$.

Let $\pi$ denote the stationary distribution of $\xi_t = (a_t,y_t)$ and recall that $P_{t-1}^{t'}$ denotes the conditional distribution of $\xi_{t'}$ conditioned on $\mathcal{F}_{t-1}$. 
We verify Assumption~\ref{assmp:mixing_norm_revised} using the following steps:
\begin{align*}
    &\quad\|F(x)-\EE[\widetilde{F}(x,\xi_{t'})|\Fspace_{t-1}]-F(y)+\EE[\widetilde{F}(y,\xi_{t'})|\Fspace_{t-1}]\|  \\ &= \|\EE_{\xi \sim \pi}[\widetilde{F}(x,\xi)-\widetilde{F}(y,\xi)]-\EE[\widetilde{F}(x,\xi_{t'})-\widetilde{F}(y,\xi_{t'})|\Fspace_{t-1}]\| \\ 
    &= \left\|\int_{\xi \in \Xi}(\widetilde{F}(x,\xi)-\widetilde{F}(y,\xi))(d\pi(\xi) - dP_{t-1}^{t'}(\xi))\right\|.
\end{align*}
 Using the Lipschitz property of $\widetilde{F}$ and the definition of the total variation distance, we obtain
\begin{align*}
    &\quad\|F(x)-\EE[\widetilde{F}(x,\xi_{t'})|\Fspace_{t-1}]-F(y)+\EE[\widetilde{F}(y,\xi_{t'})|\Fspace_{t-1}]\| \\ &\leq 2\widetilde{L}_1 \|x-y\| \TV(\Pi,P_{t-1}^{t'}) \overset{\1}{\leq} 2\widetilde{L}_1 C_{TV}\cdot  \rho_{TV}^{t'-t} \cdot \|x-y\|,
\end{align*}
where step $\1$ again follows from Assumption~\ref{assmp:TV_mixing}. Thus for GLMs Assumption~\ref{assmp:mixing_norm_revised} is satisfied with $\rho_2= \rho_{TV}$ and $\cp = 2\widetilde{L}_1 C_{TV}$.

Since all these assumptions hold, invoking Theorem~\ref{thm:main_error} yields a proof of Corollary~\ref{corollary:GLM_1}.

\subsection{Proof of Corollary~\ref{corollary:GLM_2}}\label{sec:pf_corollary_GLM_2}
Theorem~\ref{thm:emulate_iid} requires the stochastic operator $\widetilde{F}$ to be Lipschitz in its second argument, i.e., the covariates and observations $(a_t,y_t)$ for GLMs.

Recall $D_f$ is the bound for the link function $f$ from Ineq.~\eqref{eq:f_bound}. We also recall that $a_t \leq D_a$ almost surely for all $t$ and $X$ has a diameter at most $D_x$.

We have
\begin{align}
\nonumber
    \|\widetilde{F}(x,\xi_1)-\widetilde{F}(x,\xi_2)\| &= \|a_1 f(a_1^T x)-a_1 y_1-a_2 f(a_2^T x)+a_2 y_2\| \\ \nonumber
    &\leq  \|a_1 f(a_1^T x)-a_2 f(a_2^T x)\|+\|a_1 y_1-a_2 y_2\| \\ \nonumber
    &= \|a_1 f(a_1^T x)-a_1 f(a_2^T x)+a_1 f(a_2^T x)-a_2 f(a_2^T x)\| \\ \nonumber
    & \qquad+\|a_1 y_1-a_1 y_2+a_1 y_2-a_2 y_2\| \\ \nonumber
    &\leq L_{f}\, \|a_1\|\, \|x\| \, \|a_1-a_2\|+ |f(a_2^T x)|\,\|a_1-a_2\| \\ \label{eq:intermediate_lipschitz_bound}
    & \qquad+ \|a_1\|\,|y_1-y_2| + |y_2|\,\|a_1-a_2\|. \\ \nonumber
\end{align}
Substituting $|f(a^T x)| \leq D_f$ along with the assumption that $\|x_t\| \leq D_x$ and $|y_t| \leq D_y$ for all $t$ into Ineq. \eqref{eq:intermediate_lipschitz_bound}, we obtain
\begin{align}
    \|\widetilde{F}(x,\xi_1)-\widetilde{F}(x,\xi_2)\| &\leq \left( L_f D_a D_x + D_f + D_y\right)\|a_1-a_2\|+D_a |y_1 - y_2|.
\end{align}
This proves that the stochastic operator $\widetilde{F}$ 
is Lipschitz in its second argument with Lipschitz constant $ \widetilde{L}_2 = \sqrt{2} \max (L_f D_a D_x + D_f + D_y , D_a)$. 

Having verified the necessary assumption, invoking Theorem~\ref{thm:emulate_iid} proves Corollary~\ref{corollary:GLM_2}.

\subsection{Proof of Corollary~\ref{corollary:RL_1}}\label{sec:pf_corollary_RL_1}
Using the assumption of bounded feature vector and bounded rewards, we can show that, for the policy evaluation problem in RL, the stochastic operator defined in Eq.~\eqref{eq:RL_operator} is $(1+\gamma)\cdot D_{\psi}^2$-Lipschitz continuous in its first argument:
\begin{align*}
    \left\| \widetilde{F}(\theta, \xi_k)- \widetilde{F}(\theta', \xi_k)\right\| &= \left\|\left( \langle \psi(s_k),\theta - \theta' \rangle -\gamma \langle\psi(s_{k+1}),\theta-\theta'  \rangle \right)\psi(s_k)\right\| \\ 
    &=  \left\|\langle \psi(s_k)-\gamma \psi(s_{k+1}),\theta - \theta' \rangle  \psi(s_k)\right\|\\
    &\overset{\1}{\leq} \left\|\psi(s_k)\right\| \left\|\psi(s_k)-\gamma \psi(s_{k+1})\right\|\left\|\theta-\theta'\right\| \\ 
    &\overset{\2}{\leq} (1+\gamma) \cdot D_{\psi}^2 \cdot \|\theta - \theta'\|,
\end{align*}
where step $\1$ follows from Cauchy--Schwarz inequality and step $\2$ follows from boundedness.
Thus, we have $\bar{L} = L + \widetilde{L}_1 = (1+\gamma) (D_{\psi}^2 + \lambda_{max}(Q)) $. 

Assumption~\ref{assmp:state_dependent_noise} is verified as follows:
\begin{align*}
    &\quad \EE\left[ \big\| \widetilde{F}(\theta,\xi_{t'})-\EE [ \widetilde{F}(\theta,\xi_{t'}) | \Fspace_{t-1}] \big\|^2 \big| \Fspace_{t-1} \right] \\
    &= \EE\left[ \big\| \widetilde{F}(\theta,\xi_{t'})-\widetilde{F}(\bar{\theta},\xi_{t'})+\widetilde{F}(\bar{\theta},\xi_{t'})-\EE [ \widetilde{F}(\bar{\theta},\xi_{t'}) | \Fspace_{t-1}]+\EE [ \widetilde{F}(\bar{\theta},\xi_{t'}) | \Fspace_{t-1}]-\EE [ \widetilde{F}(\theta,\xi_{t'}) | \Fspace_{t-1}] \big\|^2 \big| \Fspace_{t-1} \right] \\ 
    &\overset{\1}{\leq} 3\EE\left[ \big\| \widetilde{F}(\theta,\xi_{t'})-\widetilde{F}(\bar{\theta},\xi_{t'})\big\|^2 \big| \Fspace_{t-1} \right] + 3\EE\left[ \big\| \widetilde{F}(\bar{\theta},\xi_{t'})-\EE [ \widetilde{F}(\bar{\theta},\xi_{t'}) | \Fspace_{t-1}]\big\|^2 \big| \Fspace_{t-1} \right] + \\ 
    &\qquad 3\EE\left[ \big\| \EE [ \widetilde{F}(\bar{\theta},\xi_{t'}) | \Fspace_{t-1}]-\EE [ \widetilde{F}(\theta,\xi_{t'}) | \Fspace_{t-1}]\big\|^2 \big| \Fspace_{t-1} \right] \\
    &\overset{\2}{\leq} 6 (1+\gamma)^2 D_{\psi}^4 \|\theta - \bar{\theta}\|^2 +  3\EE\left[ \big\| \widetilde{F}(\bar{\theta},\xi_{t'})-\EE [ \widetilde{F}(\bar{\theta},\xi_{t'}) | \Fspace_{t-1}]\big\|^2 \big| \Fspace_{t-1} \right] \\
    &\overset{\3}{\leq} 6 (1+\gamma)^2 D_{\psi}^4 \|\theta - \bar{\theta}\|^2 +  6\EE\left[ \big\| \widetilde{F}(\bar{\theta},\xi_{t'})\big\|^2 \big| \Fspace_{t-1} \right]+6\EE\left[ \big\|\EE [ \widetilde{F}(\bar{\theta},\xi_{t'}) | \Fspace_{t-1}]\big\|^2 \big| \Fspace_{t-1} \right] ,
\end{align*}
where step $\1$ and step $\3$ follow from the Young's inequality and step $\2$ from the Lipschitz property of $\widetilde{F}$. We now bound the last term in the above inequality as follows:
\begin{align*}
    \|\widetilde{F}(\bar{\theta},\xi_{t'})\|^2 &= \left\|(\langle \psi(s_{t'})-\gamma \psi(s_{t'+1}),\bar{\theta}  \rangle - R(s_{t'},s_{t'+1}))  \psi(s_{t'})\right\|^2 \\ 
    &\overset{\1}{\leq} 2\|\psi(s_{t'})-\gamma \psi(s_{t'+1})\|^2 \|\bar{\theta}\|^2 \|\psi(s_{t'})\|^2 + 2 |R(s_{t'},s_{t'+1}))|^2\|\psi(s_{t'})\|^2 \\
    &\overset{\2}{\leq} 2(1+\gamma)^2 D_{\psi}^4 D^2 + 2\bar{R}^2 D_{\psi}^2,
\end{align*}
where step $\1$ follows from the Young's inequality and Cauchy-Schwarz inequality, and step $\2$ follows from boundedness and Assumption~\ref{assmp:optimal_ball}. Using the above bound we have
\begin{align*}
    &\quad \EE\left[ \big\| \widetilde{F}(\theta,\xi_{t'})-\EE [ \widetilde{F}(\theta,\xi_{t'}) | \Fspace_{t-1}] \big\|^2 \big| \Fspace_{t-1} \right] \\
    &\overset{}{\leq} 6 (1+\gamma)^2 D_{\psi}^4 \|\theta - \bar{\theta}\|^2 +  24(1+\gamma)^2 D_{\psi}^4 D^2+ 24\bar{R}^2 D_{\psi}^2.
\end{align*}
Thus Assumption~\ref{assmp:state_dependent_noise} is satisfied with $\sigma^2 = 48(1+\gamma)^2 D_{\psi}^4 D^2+ 48 D_{\psi}^2 \bar{R}^2$ and $\zeta^2 = 12 (1+\gamma)^2 D_{\psi}^4$.
Assumption~\ref{assmp:optimal_condition} and Assumption~\ref{assmp:mixing_norm_revised} are verified in Lemma 1 and Lemma 2 of  \citet{li2023accelerated}, respectively.
We thus invoke Theorem~\ref{thm:main_error} to complete the proof of Corollary~\ref{corollary:RL_1}.

\subsection{Proof of Corollary~\ref{corollary:RL_2}}\label{sec:pf_corollary_RL_2}
We now proceed to verify assumptions related to Theorem~\ref{thm:emulate_iid} for policy evaluation. In order to verify Assumption~\ref{assmp:lipscitz2}, we need to prove that the iterates in policy evaluation remain bounded. We state the following lemma, which will be proved later in Section~\ref{sec:pf_lemma_theta_bound}.
\begin{lemma}
    \label{lemma:theta_bound}
    Suppose that the feature vector is bounded, i.e., for all $s \in \mathcal{S}, \|\psi(s)\| \leq D_{\psi}$ and the reward function is bounded, i.e., for all $s,s' \in \mathcal{S}$ $R(s,s') \leq \bar{R}$ and the step size satisfies $\eta \leq \frac{1}{T (1+\gamma) \cdot D_{\psi}^2}$. 
    Then we have
    \begin{align*}
        \|\theta_T\| \leq \left( 1+ \sqrt{d} \right)\left(\|\theta_0\| + \frac{\bar{R}}{(1+\gamma) \cdot D_{\psi}} \right).
    \end{align*}
\end{lemma}
For notational convenience, we define $D_{\theta} \defn \left( 1+ \sqrt{d} \right)\left(\|\theta_0\| + \frac{\bar{R}}{(1+\gamma) \cdot D_{\psi}} \right)$.
We now state the lemma that verifies Assumption~\ref{assmp:lipscitz2} for policy evaluation in RL, which we prove later in  Section~\ref{sec:pf_lemma_RL_lipschitz2}.
\begin{lemma}\label{lemma:RL_lipschitz2}
Under the premise of Lemma~\ref{lemma:theta_bound}, 
    we have that the stochastic operator $\widetilde F$ is Lipschitz in its second argument with the Lipschitz constant $\sqrt{3}\left((2+\gamma)\cdot D_{\psi} \cdot D_{\theta} +\bar{R}\right)$.
\end{lemma}

Having verified these assumptions, we invoke Theorem~\ref{thm:emulate_iid} to complete the proof of Corollary~\ref{corollary:RL_2}.

We now proceed to the proofs to Lemma~\ref{lemma:theta_bound} and Lemma~\ref{lemma:RL_lipschitz2}.

\subsubsection{Proof of Lemma \ref{lemma:theta_bound}}\label{sec:pf_lemma_theta_bound}
We have
\begin{align}
\nonumber
\widetilde{F}(\theta_t, \xi_t) &= \left( \langle \psi(s_t),\theta_t \rangle - R(s_t,s_{t+1})-\gamma \langle\psi(s_{t+1}),\theta_t  \rangle \right)\psi(s_t) \\ \label{eq:stochastic_operator_RL}
&= \psi(s_t)\psi(s_t)^T \theta_t - \gamma \psi(s_t)\psi(s_{t+1})^T \theta_t - \psi(s_t) R(s_t,s_{t+1}). 
\end{align}

We aim to find the fixed point $\bar{\theta} \in \mathbb{R}^d$ of the projected Bellman equation \eqref{eq:projected_bellman_matrix}. This is an unbounded problem, and the update in Eq. \eqref{eq:constrained_update} reduces to the unconstrained update of Eq. \eqref{eq:unconstrained_update}. Substituting the stochastic operator expression of Eq. \eqref{eq:stochastic_operator_RL} into Eq. \eqref{eq:unconstrained_update} we have
\begin{align}
    \nonumber
    \theta_{t+1} &= \theta_t - \eta \left( \psi(s_t)\psi(s_t)^T \theta_t - \gamma \psi(s_t)\psi(s_{t+1})^T \theta_t - \psi(s_t) R(s_t,s_{t+1})  \right).
\end{align}
Let us define $\Delta_t \defn \theta_t - \theta_0$. We thus have
\begin{align}
    \nonumber
    \Delta_{t+1} &= \Delta_t -\eta   \left( \psi(s_t) \psi(s_t)^T - \gamma \psi(s_t) \psi(s_t)^T \right)\Delta_t + \eta \psi(s_t) R(s_t,s_{t+1}) \\ \nonumber
    & \qquad - \eta \left(  \psi(s_t) \psi(s_t)^T - \gamma \psi(s_t) \psi(s_{t+1})^T \right)\theta_0 \\ \nonumber
    &= \left(I -\eta   \psi(s_t) \psi(s_t)^T + \eta \gamma \psi(s_t) \psi(s_t)^T \right)\Delta_t + \eta \psi(s_t) R(s_t,s_{t+1}) \\ \nonumber
    & \qquad - \eta \left(  \psi(s_t) \psi(s_t)^T - \gamma \psi(s_t) \psi(s_{t+1})^T \right)\theta_0.
\end{align}
Taking the $l_{\infty}$ norm and using the non-expansive property of $\left( I -\eta \psi(s_t) \psi(s_t)^T + \eta \gamma \psi(s_t) \psi(s_t)^T \right)$ in the $l_{\infty}$ norm we obtain
\begin{align}
    \nonumber
    \linfnorm{\Delta_{t+1}} \leq \linfnorm{\Delta_t} + \eta D_{\psi} \bar{R} + \eta (1+\gamma)D_{\psi} ^2\linfnorm{\theta_0}.
\end{align}
Summing from $t=0$ to $T-1$ we get
\begin{align}
    \nonumber
    \linfnorm{\Delta_{T}} \leq T \eta D_{\psi} \bar{R} +T \eta (1+\gamma)D_{\psi}^2 \linfnorm{\theta_0}.
\end{align}
Converting the $l_{\infty}$ to $\ell_2$ norm and using the definition of step size $\eta \leq \frac{1}{T\widetilde{L}_1}$  we have
\begin{align}
    \nonumber
    \|\Delta_{T}\| \leq \frac{\sqrt{d}\cdot D_{\psi}\cdot\left(\bar{R} + (1+\gamma)D_\psi\|\theta_0\|\right)}{\widetilde{L}_1}.
\end{align}
This completes the proof of Lemma~\ref{lemma:theta_bound}.
\subsubsection{Proof of Lemma~\ref{lemma:RL_lipschitz2}}\label{sec:pf_lemma_RL_lipschitz2}
We have
\begin{align}
\nonumber
    \left\| \widetilde{F}(\theta, \xi_{t})- \widetilde{F}(\theta, \xi_{t'})\right\|
    &\leq \underbrace{\left\| \langle \psi(s_t),\theta \rangle \psi(s_t) - \langle \psi(s_{t'}),\theta \rangle \psi(s_{t'})\right\|}_{E_1} \\ \nonumber
    &\qquad + \underbrace{ \left \| R(s_t,s_{t+1})\psi(s_{t}) -R(s_{t'},s_{t'+1})\psi(s_{t'}) \right\|}_{E_2} \\ \label{eq:RL_lipschitz_bound}
    &\qquad + \gamma \underbrace{\left \| \langle \psi(s_{t+1}),\theta \rangle \psi(s_t) - \langle \psi(s_{t'+1}),\theta \rangle \psi(s_{t'}) \right\|}_{E_3}.
\end{align}
Now we bound the terms $E_1$, $E_2$, and $E_3$ individually.
\begin{align} \nonumber
    E_1 &= \left\| \langle \psi(s_t),\theta \rangle \psi(s_t) - \langle \psi(s_{t'}),\theta \rangle \psi(s_{t'})\right\|\\
    &\nonumber\leq \| \langle \psi(s_t),\theta \rangle \psi(s_t)-\langle \psi(s_{t'}),\theta \rangle \psi(s_{t})\|  +\|\langle \psi(s_{t'}),\theta \rangle \psi(s_{t}) - \langle \psi(s_{t'}),\theta \rangle \psi(s_{t'})\| \\ \nonumber
    & =  \| \langle \psi(s_t)-\psi(s_{t'}),\theta \rangle \psi(s_t)\|  +\|\langle \psi(s_{t'}),\theta \rangle (\psi(s_{t})-\psi(s_{t'})) \| \\ \label{eq:R1}
    &\leq 2 \cdot D_{\psi} \cdot \|\theta\| \cdot \|\psi(s_t)-\psi(s_{t'})\|.
\end{align}
For $E_2$ we have
\begin{align}
\nonumber
    E_2 &= \left \| R(s_t,s_{t+1})\psi(s_{t}) -R(s_{t'},s_{t'+1})\psi(s_{t'}) \right\|\\
    \nonumber &\leq \|R(s_t,s_{t+1})\psi(s_{t}) -R(s_{t},s_{t+1})\psi(s_{t'})\| + \|R(s_t,s_{t+1})\psi(s_{t'}) -R(s_{t'},s_{t'+1})\psi(s_{t'})\| \\ \nonumber
    &= \|R(s_t,s_{t+1})(\psi(s_{t}) -\psi(s_{t'}))\| + \|(R(s_t,s_{t+1}) -R(s_{t'},s_{t'+1}))\psi(s_{t'})\| \\ 
    &\leq \bar{R} \cdot \|\psi(s_{t}) -\psi(s_{t'}\|  + D_{\psi} \cdot |(R(s_t,s_{t+1}) -R(s_{t'},s_{t'+1})|.\label{eq:R2}
\end{align}
For $E_3$ we have
\begin{align}
\nonumber
    E_3 &= \gamma \left \| \langle \psi(s_{t+1}),\theta \rangle \psi(s_t) - \langle \psi(s_{t'+1}),\theta \rangle \psi(s_{t'}) \right\|\\
    \nonumber &\leq \gamma  \| \langle \psi(s_{t+1}),\theta \rangle \psi(s_t)-\langle \psi(s_{t'+1}),\theta \rangle \psi(s_t)\|  + \gamma \|\langle \psi(s_{t'+1}),\theta \rangle \psi(s_t) - \langle \psi(s_{t'+1}),\theta \rangle \psi(s_{t'}) \| \\ \nonumber 
    &\leq \gamma\| \langle \psi(s_{t+1})-\psi(s_{t'+1}),\theta \rangle \psi(s_t)\|  + \gamma \|\langle \psi(s_{t'+1}),\theta \rangle (\psi(s_t)-\psi(s_{t'}) ) \| \\ 
    &\leq \gamma D_{\psi} \cdot \|\theta\| \cdot \|\psi(s_{t+1})-\psi(s_{t'+1}\|  + \gamma D_{\psi} \cdot \|\theta\| \cdot \|\psi(s_t)-\psi(s_{t'})\|.\label{eq:R3}
\end{align}
Substituting the bounds in \eqref{eq:R1}, \eqref{eq:R2} and \eqref{eq:R3} into the Ineq.~\eqref{eq:RL_lipschitz_bound} and using Lemma \ref{lemma:theta_bound} we obtain
\begin{align}
    \nonumber
    \left\| \widetilde{F}(\theta, \xi_{t})- \widetilde{F}(\theta, \xi_{t'})\right\| &\leq \left( (2+\gamma)\cdot D_{\psi} \cdot D_{\theta} +\bar{R} \right) \cdot \|\psi(s_t)-\psi(s_{t'})\| \\ \nonumber
    &\qquad + \gamma D_{\psi} \cdot D_\theta \cdot \|\psi(s_{t+1})-\psi(s_{t'+1}\| + D_{\psi} \cdot |(R(s_t,s_{t+1}) -R(s_{t'},s_{t'+1})|.
\end{align}
This proves that the stochastic operator $\widetilde{F}$ in policy evaluation is Lipschitz in its second argument with the Lipschitz constant $\widetilde{L}_2 = \sqrt{3} \left( (2+\gamma)\cdot D_{\psi} \cdot D_{\theta} +\bar{R} \right)$.

This concludes the proof of Lemma~\ref{lemma:RL_lipschitz2} and thus of Corollary~\ref{corollary:RL_2}.

\section{Discussion} \label{sec:discussion}

We introduced the MER algorithm, a novel approach for solving stochastic VIs with Markovian data when we have access to a memory buffer. 
We provided an upper bound on its stochastic error, showing a faster decay than the serial method based on vanilla stochastic approximation. In addition, we provided a two-sided expectation error bound with respect to the i.i.d. setting to demonstrate that MER effectively mitigates the challenges posed by data dependence, allowing the algorithm to mimic and closely track the performance of stochastic approximation methods designed for i.i.d. samples. A key advantage of our approach is that it achieves these properties without requiring prior knowledge of the Markov chain's mixing time, a common but often restrictive requirement in related work. We also provided numerical results to corroborate the effectiveness of the MER algorithm in solving practical problems such as signal estimation in GLMs and policy evaluation in RL. More broadly, the MER algorithm can be viewed as a decorrelation device that provides theoretical justification for the experience replay heuristic. We believe this device could be leveraged to solve broader classes of problems involving dependent data.

Our work suggests several salient directions for future research, and we mention some concrete technical questions below. First, the current MER algorithm only achieves a suboptimal deterministic error. It would be interesting to incorporate acceleration techniques for variational inequalities (e.g., \citet{kotsalis2022simple, tianjiao2022}) to match the deterministic error lower bound \citep{nemirovsky1992information}. Second, though order-optimal, the stochastic error of the MER method can be potentially sharpened beyond worst-case guarantees. For example, instance-dependent lower bounds are known for solving policy evaluation and more general fixed-point equations \citep{li2023accelerated, mou2023optimal}, and these lower bounds have been achieved by stochastic approximation methods with variance reduction~\citep{li2023accelerated} as well as sample average approximation~\citep{pananjady2020instance}. An interesting direction for future work is to develop instance-dependent information-theoretic lower bounds for the Markovian setting with a memory buffer, and to design efficient MER-type algorithms that achieve these bounds.

\subsection*{Acknowledgments}

This work was supported in part by the National Science Foundation under grants CCF-2107455 and DMS-2210734, a Google Research Scholar award, and by research awards/gifts from Adobe,
Amazon, and Mathworks.

\small
\bibliographystyle{abbrvnat}
\bibliography{reference-template}

\end{document}